\pdfoutput=1
\documentclass{amsart}

\usepackage[mathscr]{eucal}
\usepackage{amssymb}
\usepackage[usenames,dvipsnames]{xcolor} 
\usepackage[normalem]{ulem}
\usepackage{wasysym}
\usepackage{amsthm}
\usepackage{bbold}
\usepackage{comment}
\usepackage{enumitem}
\usepackage{amsmath}
\usepackage{tikz-cd}
\usepackage{stmaryrd} 
\usepackage{etoolbox} 
\usepackage{microtype}
\usepackage{adjustbox}
\usepackage{mathtools}
\usepackage{lipsum}
\usepackage[unicode]{hyperref} 

\definecolor{dark-red}{rgb}{0.5,0.15,0.15}
\definecolor{dark-blue}{rgb}{0.15,0.15,0.6}
\definecolor{dark-green}{rgb}{0.15,0.6,0.15}

\hypersetup{
    colorlinks, linkcolor=OliveGreen,
    citecolor=OliveGreen, urlcolor=OliveGreen
}

\usepackage[nameinlink,capitalise,noabbrev]{cleveref}

\numberwithin{equation}{section}
\usepackage[all]{xy}
\xyoption{line}
\usepackage{graphicx}
\usepackage{caption}

\newtheorem{thmx}{Theorem}

\newtheorem{Thm}[equation]{Theorem}
\newtheorem*{Thm*}{Theorem}
\newtheorem*{MainThm*}{Main Theorem}
\newtheorem{Prop}[equation]{Proposition}
\newtheorem{Lem}[equation]{Lemma}
\newtheorem{Cor}[equation]{Corollary}

\newtheorem{question}[equation]{Question}

\newtheorem*{Que*}{Question}

\theoremstyle{remark}
\newtheorem{Def}[equation]{Definition}

\newtheorem{Not}[equation]{Notation}
\newtheorem{Exa}[equation]{Example}

\newtheorem{Cons}[equation]{Construction}
\newtheorem{Conv}[equation]{Convention}
\newtheorem{Hyp}[equation]{Hypothesis}
\newtheorem{Rec}[equation]{Recollection}

\newtheorem{Rem}[equation]{Remark}

\tikzset{
    labelrotatebelow/.style={anchor=north, rotate=90, inner sep=1.0mm}
}
\tikzset{
    labelrotateabove/.style={anchor=south, rotate=90, inner sep=1.0mm}
}
\SelectTips{cm}{10}

\newcommand{\nc}{\newcommand}
\nc{\dmo}{\DeclareMathOperator}

\usepackage{todonotes}

\nc{\overbar}[1]{\mkern 1.5mu\overline{\mkern-1.5mu#1\mkern-1.5mu}\mkern 1.5mu}
\nc{\WCat}{\mathcal{W}\mathcal{C}\mathrm{at}_\infty}
\nc{\TCat}{\mathcal{S}\mathcal{C}\mathrm{at}_\infty}

\nc{\uMod}{\underline{\mathrm{Mod}}}
\nc{\umod}{\underline{\mathrm{mod}}}
\nc{\kappaaux}{g}
\nc{\kappaCh}{{\kappaaux(\cat C_h)}}
\nc{\kappam}{{\kappaaux({\mathfrak m})}}
\nc{\kappaP}{\Gamma_{\cat P}\unit}
\nc{\kappaQ}{{\kappaaux(\cat Q)}}
\nc{\kappaCP}{{\kappaaux_{\cat C}(\cat P)}}
\nc{\kappaDP}{{\kappaaux_{\cat D}(\cat P)}}
\nc{\kappaCQ}{{\kappaaux_{\cat C}(\cat Q)}}
\nc{\kappaDQ}{{\kappaaux_{\cat D}(\cat Q)}}
\nc{\kappaphiB}{{\kappaaux(\phi(\cat B))}}
\nc{\kappaphiQ}{{\kappaaux(\varphi(\cat Q))}}

\dmo{\Sub}{Sub}
\dmo{\Proj}{Proj}
\dmo{\LMod}{LMod}
\dmo{\cell}{cell}
\nc{\Prst}{{\cat P}\mathrm{r^{st}}}
\nc{\Mack}[2]{\mathrm{Mack}_{#1}(#2)}
\dmo{\fin}{{fin}}
\dmo{\Sphere}{\mathbb{S}}
\dmo{\Alg}{Alg}
\dmo{\CAlg}{CAlg}
\nc{\HA}{{\rmH \hspace{-0.2em}\bbA}}
\nc{\HZ}{{\rmH \hspace{-0.2em}\bbZ}}
\nc{\HZbar}{{\rmH \hspace{-0.2em}\underline{\bbZ}}}
\nc{\Fp}{{\bbF_{\hspace{-0.1em}p}}}
\nc{\HFp}{{\rmH \hspace{-0.15em}\bbF_{\hspace{-0.1em}p}}}

\nc{\mathfrakp}{\mathfrak{p}}
\nc{\mathfrakq}{\mathfrak{q}}
\nc{\mathfrakS}{\mathfrak{S}}
\nc{\mathfrakT}{\mathfrak{T}}
\nc{\Z}{\mathbb{Z}}
\nc{\cF}{\mathcal{F}}
\nc{\hspec}[1]{\Spc^\mathrm{h}({#1})}

\nc{\DPerm}{\mathbf{DPerm}}

\dmo{\Id}{Id}
\dmo{\Loc}{Loc}
\dmo{\Spc}{Spc}
\dmo{\aSpc}{aSpc}
\dmo{\thick}{thick}
\nc{\thickt}[1]{\thick_\otimes\langle #1 \rangle}
\nc{\radthickt}[1]{\thick_\otimes^{\sqrt{}}\langle #1 \rangle}
\nc{\Loct}[1]{\Loc_\otimes\langle #1 \rangle}
\nc{\Thickt}[1]{\mathrm{Thick}_\otimes\langle #1 \rangle}
\dmo{\perf}{perf}
\dmo{\End}{End}
\dmo{\Mor}{Mor}
\dmo{\Hom}{Hom}
\dmo{\id}{id}
\dmo{\im}{im}
\dmo{\Ker}{Ker}
\dmo{\ind}{ind}
\dmo{\Ind}{Ind}
\dmo{\CoInd}{coind}
\dmo{\res}{res}
\dmo{\infl}{infl}
\dmo{\triv}{triv}
\dmo{\Tel}{Tel} 
\dmo{\grMod}{grMod}%
\dmo{\Mod}{Mod}%
\dmo{\opname}{op}
\dmo{\SH}{\mathcal{S}\mathcal{H}}
\dmo{\smallb}{b}
\dmo{\Spec}{Spec}
\dmo{\MaxSpec}{MaxSpec}
\dmo{\supp}{supp}
\dmo{\Supp}{Supp}
\dmo{\cosupp}{cosupp}
\dmo{\Cosupp}{Cosupp}
\dmo{\hsupp}{hsupp}
\dmo{\Pinj}{Pinj}
\dmo{\Inj}{Inj}
\renewcommand{\mod}{\mathrm{mod}}
\dmo{\StMod}{StMod}
\dmo{\stmod}{stmod}
\dmo{\topp}{top}
\dmo{\spec}{spec}
\dmo{\ev}{ev}

\nc{\bbL}{\mathbb{L}}
\nc{\bbA}{\mathbb{A}}
\nc{\bbE}{\mathbb{E}}
\nc{\bbN}{\mathbb{N}}
\nc{\bbQ}{\mathbb{Q}}
\nc{\bbZ}{\mathbb{Z}}
\nc{\bbF}{\mathbb{F}}
\nc{\bbS}{\mathbb{S}}
\nc{\cat}[1]{\mathscr{#1}}

\nc{\cA}{\mathcal{A}}
\nc{\cB}{\mathcal{B}}
\nc{\cC}{\mathcal{C}}
\nc{\cE}{\mathcal{E}}
\nc{\cU}{\mathcal{U}}

\nc{\CB}{\mathbf{B}}

\nc{\sD}{\mathsf{D}}

\nc{\ihom}{{\underline{\hom}}}
\nc{\iHom}{\mathcal{H}\mathrm{om}}

\nc{\Mid}{\,\big|\,}
\nc{\SET}[2]{\big\{\,#1\Mid#2\,\big\}}
\nc{\unit}{\mathbb{1}}
\nc{\xra}{\xrightarrow}

\dmo{\Sp}{Sp}
\dmo{\Ho}{Ho}
\dmo{\Fin}{Fin}
\dmo{\add}{add}
\dmo{\Fun}{Fun}
\dmo{\Ext}{Ext}

\dmo{\Map}{Map}
\dmo{\Span}{Span}
\dmo{\N}{N}
\dmo{\Cat}{Cat}
\dmo{\colim}{colim}
\dmo{\hocolim}{hocolim}
\dmo{\Ch}{Ch}

\dmo{\rep}{rep}
\dmo{\Rep}{Rep}
\nc{\uRep}{\underline{\Rep}}
\nc{\urep}{\underline{\rep}}
\nc{\ustmod}{\underline{\stmod}}
\nc{\uStMod}{\underline{\StMod}}
\nc{\uDPerm}{\underline{\mathrm{DPerm}}}
\nc{\uDperm}{\underline{\mathrm{Dperm}}}
\dmo{\Ab}{Ab}
\dmo{\Set}{Set}
\dmo{\Spcl}{Spcl}
\nc{\Funadd}{\Fun_{\add}}
\dmo{\proj}{proj}
\dmo{\f}{fth}
\dmo{\cof}{cof}
\dmo{\deftensor}{Def^\otimes}
\dmo{\defcoid}{Def^{coid}}
\dmo{\Grp}{Grp}

\dmo{\dual}{dual}

\dmo{\Perf}{Perf}
\dmo{\tel}{tel}

\dmo{\rk}{rk}

\dmo{\glo}{glo}
\dmo{\Pic}{Pic}
\dmo{\gl}{gl}
\nc{\fan}[1]{\mathscr{F}_{#1}}
\dmo{\GL}{GL}
\dmo{\fib}{fib}
\dmo{\Out}{Out}
\dmo{\Fan}{Fan}
\dmo{\cons}{cons}
\dmo{\Bor}{Bor}

\newcommand{\Orb}{\mathrm{Orb}}

\newcommand{\PrL}{\mathrm{Pr}^{\mathrm{L}}_{\mathrm{st}}}
\dmo{\Ar}{Ar}

\newcounter{enum-resume-hack}
\Crefname{Thm}{Theorem}{Theorems}
\Crefname{Prop}{Proposition}{Propositions}
\Crefname{Lem}{Lemma}{Lemmas}
\Crefname{thmx}{Theorem}{Theorems}

\thanks{}

\begin{document}
\title[Descent and finite permutation resolutions]{Descent and finite permutation resolutions for discrete groups}

\author{{Juan Omar} G\'omez}
\author{Luca Pol}

\makeatletter
\patchcmd{\@setaddresses}{\indent}{\noindent}{}{}
\patchcmd{\@setaddresses}{\indent}{\noindent}{}{}
\patchcmd{\@setaddresses}{\indent}{\noindent}{}{}
\patchcmd{\@setaddresses}{\indent}{\noindent}{}{}
\makeatother

\address{Juan Omar G\'omez, Fakultat f\"ur Mathematik, Universit\"at Bielefeld, D-33501 Bielefeld, Germany}
\email{jgomez@math.uni-bielefeld.de}
\urladdr{https://sites.google.com/cimat.mx/juanomargomez/home}

\address{Luca Pol, Max Planck Institute for Mathematics, Vivatsgasse 7, 53111 Bonn, Germany}
\email{pol@mpim-bonn.mpg.de}
\urladdr{https://sites.google.com/view/lucapol/}

\maketitle

\begin{abstract}
    Let $G$ be a discrete group with a finite-dimensional model for the classifying space for proper actions, and let $k$ be a commutative Noetherian ring of finite global dimension. In this setting, we prove that the homotopy category of projective $kG$-modules, the stable module category of $kG$-modules as defined by Mazza-Symonds, and the derived category of permutation $kG$-modules with finite isotropy, admit descent to finite subgroups. As an application, we show that any $kG$-module of type $FP_\infty$ is a retract of a module that admits a finite resolution by finitely generated $\natural$-permutation modules with finite isotropy, generalizing a result of Balmer-Gallauer.
\end{abstract}
\tableofcontents
\setcounter{tocdepth}{1}

\setlength{\parindent}{0cm}
\setlength{\parskip}{0.8ex}

\section{Introduction}
Let $G$ be a group and $k$ a commutative ring. Recall that a permutation $kG$-module is the $k$-linearization of a $G$-set, and that a $\natural$-permutation module is a module that arises as a retract of a permutation module. These modules occupy a distinguished place in the representation theory of groups, not only because of their central role in the representation theory of finite groups \cite{BC21}, but also to their interactions with other areas of mathematics; for instance, in the theory of Artin motives \cite{Voe00} and in the study of cohomological Mackey functors, see \cite{BG23} and the references therein.

Our interest in this class of modules is motivated by a striking result of Balmer and Gallauer \cite{balmer2023finite}, which establishes that for finite groups and regular rings, every finitely generated module is a retract of a module admitting a finite resolution by finitely generated $\natural$-permutation modules. This is surprising because permutation modules, with their comparatively rigid structure, may not appear ‘wild
enough’ to control all $kG$-modules. 

A natural question is whether this result can be extended to a broader class of groups. In this direction, groups admitting a finite-dimensional model for the classifying space for proper actions provide a promising framework for exploring the representation theory of infinite groups; see for instance \cite{Kro93}, \cite{Ben97}, \cite{Talelli}, \cite{CEKT14}, \cite{mazza2019stable} and \cite{Gomez24}. This is a large class of groups, encompassing, among others,  arithmetic groups, automorphisms groups of free groups, groups with finite virtual cohomological dimension and countable locally finite groups. However, in this setting the notion of finitely generated modules is no longer adequate, since the group algebra of an infinite group is seldom Noetherian. One is thus led to consider instead modules of type $FP_\infty$ (also known as pseudo-coherent modules), namely those that admit a resolution by finitely generated projective modules; see \cite[Section 9.3]{krause} for further discussion of their importance. Among $\natural$-permutation modules, an important subclass is formed by the \emph{finite $\natural$-permutation modules}, namely those that are retracts of a finite direct sum of modules of the form $k(G/H)$, where $H$ is a finite subgroup of $G$. Our generalization of Balmer-Gallauer result states the following, which corresponds to \cref{coro-finite-perm-res}.

\begin{thmx}
    Let $G$ be a group with a finite-dimensional model for the classifying space for proper actions and $k$ be a Noetherian ring of finite global dimension. Then any $kG$--module of type $FP_\infty$ is a retract of a module that admits a finite resolution by  finite $\natural$-permutation modules. 
\end{thmx}
At this point, let us emphasize that our proof differs substantially from that of Balmer and Gallauer in the finite case. Their argument relies on techniques specific to finite groups—such as tensor induction from the trivial subgroup—whose extension to infinite groups is far from clear. Rather than attempting to adapt their methods to our setting, we follow the philosophy of understanding the representation theory of a group in terms of the representation theory of its finite subgroups, thereby allowing us to leverage the result of \cite{balmer2023finite}. This perspective can be made precise by working in the $\infty$-categorical setting, which provides a framework in which one can apply descent techniques inspired by homotopy theory \cite{Mat16}, \cite{MNN2017}. Such techniques are not accessible when working solely with triangulated categories.

We now describe our methods in more detail.

\subsection{Methodology}  Our starting point is \cite[Theorem 1.4]{balmer2023finite}, from which the aforementioned result of Balmer and Gallauer is deduced. It can be interpreted as follows: for a finite group $G$ and a regular commutative ring $k$, there is a finite localization 
\begin{equation}\label{fin-local}
\DPerm(G,k) \to \mathbf{K}(\Proj(kG))
\end{equation}
between the homotopy category of projective $kG$-modules and the big derived category of permutation $kG$-modules. The former provides a suitable \textit{big} version of the bounded derived category of finite $kG$-modules \cite{BBIKP}, while the latter serves as a corresponding \textit{big} version of the bounded homotopy category of finite $\natural$-permutation modules.

The formulation given in \eqref{fin-local} is more amenable to generalization to infinite discrete groups and can be approached using higher category theory.
With this perspective in mind, we begin by introducing suitable analogues of these categories for infinite groups, with the aim of describing them in terms of the corresponding categories associated to their finite subgroups. It turns out that relatively direct (\textit{ad hoc}) constructions suffice, at least for well-behaved groups.

In particular, we consider the \textit{derived category of finite permutation modules} $\DPerm^{\fin}(G,k)$, defined as the localizing subcategory of the homotopy category of $kG$-modules generated by the finite permutation modules; this construction is carried out in \cref{sec-perm}. We also study the homotopy category $\mathbf{K}(\Proj(kG))$ of projective $kG$-modules in \cref{sec-kproj}. All of these categories are compactly generated, at least if $G$ admits a finite-dimensional model for the classifying space for proper actions.

A key observation is that these categories are functorial in the group variable with respect to group monomorphisms, which naturally leads us to consider a more general framework.

\subsection{Descent}
Fix a group $G$. 
We are interested in functors
\[
\cat C \colon \Orb_G^{\opname}\to \PrL
\]
from the $G$-orbit $\infty$-category to the $\infty$-category  $\PrL$  of presentable stable $\infty$-categories and left adjoint functors. We furthermore assume that 
 \begin{enumerate}
     \item For any subgroup $H\subseteq G$, the restriction functor $\res^G_H \colon \cat C(H) \to \cat C(K)$ admits a left adjoint $\ind_H^G$, which we call induction.
    \item The restriction-induction functors satisfy the Mackey formula.
    \end{enumerate}

The relevance of considering such functors lies in the following descent result, which appears later as \cref{thm-descent}.

\begin{thmx}\label{thmx-descent}
   Let $G$ be a group and $\cat C \colon \Orb_G^{\opname} \to \PrL$ be as above. Suppose that there exists a collection $\cat F$ of subgroups of $G$ such that $\{\res^G_H\}_{H \in \cat F}$ is jointly conservative. Then the restriction functors induce an equivalence of stable $\infty$-categories
    \[
    \cat C(G) \xrightarrow{\sim}\lim_{G/H \in \Orb_{G,\cat F}^{\opname}} \cat C(H).
    \]
\end{thmx}

In fact, the above equivalence is symmetric monoidal if the functor $\cat C$ lifts to symmetric monoidal stable $\infty$-categories. We say that a functor satisfying the hypotheses of the previous theorem \textit{admits descent with respect to the collection $\cat F$}.

The upshot is that, for a group $G$ admitting a finite-dimensional model for the classifying space for proper actions, and for a commutative Noetherian ring $k$ of finite global dimension, the following categories admit descent with respect to the family of \emph{finite} subgroups:
\begin{itemize}
    \item The stable module category of $G$ as defined in \cite{mazza2019stable}; see \cref{thm-descent-stmod}. 
    \item The homotopy category of projective $kG$--modules; see \cref{thm-descent-kproj}.
    \item The derived category of finite permutation $kG$--modules; see \cref{thm-descent-permutation}. 
\end{itemize}

We also note that these categories admit descent with respect to other families of finite subgroups. For instance, the first two admit descent with respect to the family of  elementary abelian subgroups, while the third admits descent with respect to the family of subgroups of $p$-power order (allowing the prime $p$ to vary). See \cref{sec-descent-ex} for further details.

With these results in place, we can now explain how to extend  Balmer–Gallauer result to the class of groups under consideration. 
\subsection{A finite localization}
Let $G$ be a group admitting a finite-dimensional model for the classifying space for proper actions.
For a finite subgroup $H\subseteq G$, let
    \[
    q_H \colon \DPerm(H,k) \to \mathbf{K}(\Proj(kH))
    \]
denote the finite localization functor of \cite[Theorem 1.4]{balmer2023finite}; see \cref{prop-radjoint-fincase} for a reformulation.  Using our descent results, we define a functor
\[
 q_G \colon \DPerm^{\fin}(G,k) \to \mathbf{K}(\Proj(kG)) 
\]
as the limit of the finite localizations $q_H$; see \cref{cons-q}. The language of $\infty$-categories let us formally conclude that $q_G$ is a smashing localization. Although such a limit need not, in general, be a finite localization, we show that in our situation it is, see \cref{thm-perm-localization-kproj}.
\begin{thmx}\label{thmx-finite-loc}
       Let $G$ be a group admitting a finite-dimensional model for the classifying space for proper actions, and let $k$ be a commutative Noetherian ring of finite global dimension. Then the functor 
    \[
    q_G \colon \DPerm^{\fin}(G,k) \to \mathbf{K}(\Proj(kG))
    \]
    is a finite localization. 
\end{thmx}

The final step in proving the main result of this introduction is to identify the compact objects of the category $\mathbf{K}(\Proj(kG))$. This is carried out in \cref{cor-compact-Kproj}.

\subsection{Further applications} We now explain how our descent results can be used to address classification problems in two specific contexts. These examples are, of course, chosen somewhat arbitrarily, and we emphasize that descent techniques have applications beyond those presented here.

First, we leverage \cref{thm-descent-permutation} to develop computational tools for determining the Picard group of the derived category of finite permutation modules. Since this invariant has already been studied in the finite case, we also show that our methods can be used to provide explicit computations for infinite groups. As an illustration, we provide a complete computation for the group $\mathrm{SL}_2(\mathbb Z)$; see \cref{sec-picard}.

On the other hand, we use \cref{thm-descent-stmod} to give an explicit classification of separable commutative algebras in the stable module category for groups such as $\mathrm{SL}_2(\mathbb Z)$ and $\mathbb{Z}/p^\infty$. This is achieved by exploiting the relationship between these objects and the so-called Galois group of a stable homotopy theory (see \cite{Mat16}), using methods from \cite{NaumannPol}. Details are given in \cref{sec-galois}.

\subsection{Relation to other work}We have already mentioned that our descent techniques are motivated by \cite{Mat16} and \cite{MNN2017}; in particular, the former is a pioneering work on descent methods in the representation theory of finite groups. In \textit{loc. cit.}, Mathew proves that the stable module category in the modular setting admits descent with respect to the family of elementary abelian $p$-subgroups. In \cite{Hun22}, Hunt identifies additional collections of subgroups for which the stable module category admits descent, while \cite{Gomez24} establishes analogous results for the stable module category of certain infinite groups.

We emphasize that our approach is more general: it recovers the results of \cite{Hun22} and \cite{Gomez24}, and moreover applies in a broader setting; for instance, allowing for more general ground rings, rather than just fields. 

\subsection{Outline of the document}  \cref{sec-recollections} is devoted to introducing notation and terminology concerning classes of infinite groups that will be used throughout the paper. In particular, we discuss groups of type $\Phi_k$, groups admitting a finite-dimensional model for the classifying space for proper actions, and hierarchically defined groups.

We introduce the stable module category for hierarchically defined groups in \cref{sec-stmod}, the homotopy category of projective modules for groups of type $\Phi_k$ in \cref{sec-kproj}, and the derived category of finite permutation modules in \cref{sec-perm}. In \cref{sec-functoriality}, we show that these categories are functorial in the group variable, while \cref{sec-descent} is devoted to the abstract descent result \cref{thmx-descent}.

Section \cref{sec-perm-resolutions} is concerned with the proof of \cref{thmx-finite-loc}. Finally, \cref{sec-picard} and \cref{sec-galois} address applications to Picard groups and commutative separable algebras, respectively.

\subsection{Conventions}\label{conventions} 
We finish with the conventions for our paper. 
\begin{enumerate}
    \item Throughout this paper, all groups are assumed to be discrete. By a \emph{collection of subgroups} we mean a set of subgroups closed under conjugation, while a \emph{family of subgroups} is a collection that is further closed under passage to subgroups.
    \item We reserve the letter $k$ to denote a commutative (discrete) ring, and impose additional assumptions when needed. In particular, a commutative regular ring is always assumed to be Noetherian.
    \item We let $\CAlg(\PrL)$ denote the $\infty$-category of presentably symmetric monoidal stable $\infty$-categories, and colimit preserving functors between them. This is the higher analog of the category of big tt-categories.
    \item We let $\CAlg(\Cat^{\perf}_\infty)$ denote the $\infty$-category of essentially small, idempotent complete, symmetric monoidal stable $\infty$-categories, and exact functors between them. This is the higher analog of the category of essentially small, idempotent complete tt-categories.
\end{enumerate}
\subsection{Acknowledgements}
The authors thank Bastiaan Cnossen, Martin Gallauer, Kaif Hilman, Sil Linskens, Julia Sauter and Peter Symonds for helpful discussions.

JOG is supported by the Deutsche Forschungsgemeinschaft (Project-ID 491392403 – TRR 358). LP was supported by the SFB 1085 Higher Invariants in Regensburg and by the European Research Council (ERC) under Horizon Europe (grant No.~101042990). LP is grateful to Max Planck Institute for Mathematics in Bonn for its hospitality and financial support.

\section{Recollection on infinite groups}\label{sec-recollections}

In this section, we introduce two large classes of groups: those admitting a finite-dimensional model for the classifying space of proper actions, and those of type $\Phi$ as introduced in \cite{Talelli}. We discuss examples and the relation between these two classes. Finally, we recall Kropholler's hierarchy class. 

The first class of groups we introduce will depend on the following universal space. 
\begin{Rec}
Let $G$ be a discrete group.  Then there exists a $G$-CW complex $\underline{E}G$ such that the fixed point space $\underline{E}G^H$ is contractible for every finite subgroup $H\subseteq G$, and empty for every infinite subgroup $H$. This space $\underline{E}G$ is unique up to $G$-homotopy and is called the \textit{classifying space for proper actions of $G$}.
\end{Rec}
\begin{Def}\label{def-finite-model}
We say that $G$ admits a \textit{finite model for $\underline{E}G$} if there exists a $G$-CW-complex $X$ with a finitely many $G$-equivariant cells such that $X \simeq_G \underline{E}G$. We say that $G$ admits a \textit{finite-dimensional model for $\underline{E}G$} if there exists a $G$-CW-complex  $X \simeq_G \underline{E}G$ whose $G$-equivariant cells are bounded in dimension.
\end{Def} 
\begin{Exa}
    If $G$ is torsion-free (e.g. $G=\Z$) then $\underline{E}G$ agrees with the universal contractible free $G$-space $EG$.
\end{Exa}
\begin{Exa}
    Consider the infinite dihedral group $G=D_{2^{\infty}}=\Z \rtimes \Z/2$. The real line with the action given by translation and reflection is a model for $\underline{E}G$. 
\end{Exa}
\begin{Exa}
     Other concrete examples of groups admitting finite-dimensional models for $\underline{E}G$ include arithmetic groups, automorphisms groups of free groups, groups with finite virtual cohomological dimension and locally finite groups satisfying certain mild set-theoretic conditions \cite[Theorem 2.6]{Dicks}.  
\end{Exa}
The next class of groups we introduce will depend on a choice of a commutative ring.  Unless otherwise stated, we will work under the following:
\begin{Hyp}\label{hyp}
    We let $k$ denote a commutative Noetherian ring of finite global dimension. 
\end{Hyp}

\begin{Rem}\label{rem-regular-rings}
   Recall that a commutative ring $k$ is \textit{regular} if it is Noetherian and locally of finite global dimension. In particular, if $k$ satisfies \cref{hyp}, then it is regular.
\end{Rem}

\begin{Def}\label{def-group-type-Phi}
    A group $G$ is said to be of \textit{type $\Phi_k$} if the following holds: any $kG$--module $M$ has finite projective dimension over $kG$ if and only if, for every finite subgroup $H\subseteq G$, the restriction $\res^G_H M$ has finite projective dimension over $kH$.
\end{Def}

\begin{Rem}
It is clear from the definition that the class of groups of type $\Phi_k$ is closed under taking subgroups.    
\end{Rem}

\begin{Exa}\label{ex-special-cx}
    Any group $G$ that admits  an exact complex of $kG$--modules of the form 
    \[
     0\to C_n \to C_{n-1} \to\ldots \to C_0 \to k \to 0 
    \]
    with each $C_i$ a coproduct of modules of the form $k(G/H)$ for finite subgroups $H\subseteq G$ is of type $\Phi_k$. This is proved in \cite[Proposition 2.5]{mazza2019stable}.
\end{Exa}
\begin{Exa}\label{ex-fin-dim-model-Phi}
    Let $G$ be a group admitting a finite-dimensional model $X$ for $\underline{E}G$. The augmented cellular chain complex $\overline{C}_\ast(X;k)$ of $X$  provides a complex as in \cref{ex-special-cx}. Therefore, $G$ is of type $\Phi_k$ for any $k$.

\end{Exa}

\begin{Exa}
      Groups of type $\Phi_k$ have finite finitistic dimension \cite[Lemma 2.4]{mazza2019stable}. In particular, no free abelian group of infinite rank is of type $\Phi_k$, for any choice of $k$.
\end{Exa}

\begin{Rem}\label{rem-conjecture Phi}
    In \cite{Talelli}, it was conjectured that the class of groups of type $\Phi_\mathbb{Z}$ coincides with the class of groups admitting a finite-dimensional model for the classifying space for proper actions. To the best of our knowledge, this conjecture remains open. For some of the results presented later in this document, this distinction is irrelevant; however, in other cases the arguments do rely on the existence of such a model. We will make this dependence explicit whenever necessary.
\end{Rem}

We conclude this section with another class of groups which contains properly the class of groups of type $\Phi_k$.

\begin{Rec}
    Let $\mathcal{X}$ be a class of groups. Kropholler’s hierarchical class of $\mathrm{H}\mathcal{X}$-groups is defined inductively as follows. Set $\mathrm{H}_0\mathcal{X}\coloneqq\mathcal{X}$. For each ordinal $\alpha>0$, let $\mathrm{H}_\alpha\mathcal{X}$ denote the class of groups $G$ admitting a contractible finite-dimensional $G$-CW complex such that the isotropy group of every cell belongs to $\mathrm{H}_\beta\mathcal{X}$ for some $\beta<\alpha$. Finally, define  
    \[
    \mathrm{H}\mathcal{X}\coloneqq \bigcup_{\alpha\geq 0} \mathrm{H}_\alpha\mathcal{X}.
    \]
  For further details, we refer to \cite{Kro93}.   
\end{Rec}

\begin{Rem}
   In this text, we mainly consider $\mathrm{H}\mathcal{X}$-groups where $\mathcal{X}$ is either the class of finite groups or the class of groups of type $\Phi_k$. At present, we are not aware of any examples showing that these two classes differ; see  \cref{rem-conjecture Phi}.
\end{Rem}

 \section{Stable module category}\label{sec-stmod}

In this section we introduce the stable module category of an $\mathrm{H}\Phi_k$-group following \cite{emmanouil2025class}. This construction generalizes the usual stable module category of finite groups, and Mazza-Symonds \cite{mazza2019stable} stable module category for groups of type $\Phi_k$. As always, we will be working under \cref{hyp}.

\begin{Rec}
    Let $G$ be a  group.  A complex  $P_\ast$ of projective $kG$--modules is \textit{totally acyclic} if $P_\ast$ is acyclic and $\Hom_{kG}(P_\ast,Q)$ is acyclic for any projective $kG$--module $Q$. 
    A $kG$--module $M$ is \textit{Gorenstein projective} if it is isomorphic to a kernel in a totally acyclic complex of projective $kG$--modules. See \cite[Section 3]{mazza2019stable} for further details. 
\end{Rec}

\begin{Exa}\label{ex-gorestein}
    If $G$ is finite, then a $kG$--module is Gorenstein projective if and only if it is $k$-projective; see for instance \cite[Lemma 3.5]{mazza2019stable}.
\end{Exa}

\begin{Exa}
  Let $G$ be a group. Note that a projective $kG$--module $P$ is Gorenstein projective. Indeed, the complex
    \[
    \ldots\to 0\to P \xrightarrow[]{1}P\to 0\to\ldots
    \]
    is a contractible complex of projective $kG$--modules, hence a totally acyclic complex. It follows that $P$ is isomorphic to one of the syzygies (equivalently, to the image of one of the differentials).
\end{Exa}

\begin{Rem}\label{rem-totally-acy}
   If $G$ is an $\mathrm{H}\Phi_k$-group, then any acyclic complex of projective $kG$-modules is automatically totally acyclic. Indeed, when $G$ is a group of type $\Phi_k$, this was shown in \cite[Lemma 3.14]{mazza2019stable}. The general case follows from the main result of \cite{emmanouil2025total}. Alternatively, one may deduce the statement by induction on $\mathrm{H}_\alpha \Phi_k$, following the approach of \cite[Theorem B]{DT10}. 
\end{Rem}

For an arbitrary discrete ring $R$, not necessarily commutative, we write $\Mod(R)$ to denote the category of all (left) $R$-modules.

\begin{Rec}\label{rec-GP-Frobenius}
    The subcategory $\mathrm{GP}(kG)$ of $\Mod(kG)$ on Gorenstein projective $kG$--modules is extension closed, hence it inherits an exact structure. In fact, it is well-known this structure turns $\mathrm{GP}(kG)$ into a Frobenius exact category (see \cite[Proposition 2.2]{dalezios2018quillen}) where the projective-injective objects are the projective $kG$--modules. We write $\underline{\mathrm{GP}}(kG)$ to denote the corresponding stable category.  
\end{Rec}

\begin{Rem}\label{rem-ac=stmod}
        Let $G$ be an $\mathrm{H}\Phi_k$-group. Then we can identify the stable category $\underline{\mathrm{GP}}(kG)$ with the homotopy category $\mathbf{Ac}(\mathrm{Proj}(kG))$ of acyclic complexes of projective $kG$--modules. Indeed, it is well known that $\underline{\mathrm{GP}}(kG)$ agrees with the homotopy category of totally acyclic complexes (e.g., see \cite[Theorem 4.16]{Beligiannis01012000}), and \cref{rem-totally-acy} completes the claim. 
\end{Rem}

\begin{Def}\label{rec-fib and weakequi}
Let $G$ be an $\mathrm{H}\Phi_k$-group.  Let $f\colon M\to N$ be a homomorphism of $kG$--modules. We say that $f$ is
    \begin{enumerate}
        \item  a \textit{cofibration} (resp. \textit{trivial cofibration}), if it is injective and has Gorenstein projective cokernel (resp. projective cokernel);
        \item \textit{fibration} (resp. \textit{trivial fibration}), if it is surjective (resp. surjective with kernel right $\mathrm{Ext}^1_{kG}$-orthogonal to $\mathrm{GP}(kG)$); 
       \item a \textit{weak equivalence}, if it factors as a trivial cofibration followed by a trivial fibration. 
    \end{enumerate} 
\end{Def}

\begin{Rem}\label{rem-bensoncof}
    Our assumption on the group and on the ring allows us to not distinguish between the class of Gorenstein projective $kG$-modules and the so-called class of Benson's cofibrant $kG$--modules. We refer the reader to \cite[Section 2]{emmanouil2025class}  for further details. 
\end{Rem}

\begin{Prop}
    Let $k$ be as in \cref{hyp} and $G$ an $\mathrm{H}\Phi_k$-group. The classes of fibrations, cofibrations and weak equivalences from  \cref{rec-fib and weakequi} define a  combinatorial model structure on $\Mod(kG)$ whose homotopy category is $\underline{\mathrm{GP}}(kG)$. Moreover, this is a symmetric monoidal model category if we endow $\Mod(kG)$ with the monoidal structure given by $\otimes_k$ equipped with the diagonal $G$-action. 
\end{Prop}

\begin{proof}
    The first part follows by \cite[Proposition 4.12]{emmanouil2025class}. The second claim is a particular case of \cite[Theorem 6.1]{emmanouil2025monoidal}. Alternatively, one can can adapt the argument given in the proof of \cite[Proposition 2.12]{Gomez24}. 
\end{proof}

\begin{Def}
    The \textit{stable module $\infty$-category} of $kG$--modules $\mathbf{StMod}(kG)$ is defined as the $\infty$-categorical localization of $\Mod(kG)$ at the class of weak equivalences defined above.
\end{Def}

\begin{Rem}\label{rem-stmod-sym-mon}
    The $\infty$-category $\mathbf{StMod}(kG)$ inherits a symmetric monoidal structure from the one in $\Mod(kG)$ turning it into a symmetric monoidal and stable $\infty$-category, where the tensor product commutes with colimits in each variable separately.     
\end{Rem}

\begin{Rem}\label{rem-homotopycat of stab}
   Let $G$ be a group of type $\Phi_k$. Then  the definition of $\mathbf{StMod}(kG)$ coincides with the one given in \cite{Gomez24}. In particular, we also recover Mazza-Symonds stable category as the homotopy category of $\mathbf{StMod}(kG)$.
\end{Rem}

\begin{Prop}\label{prop-induction-stmod}
    Let  $G$ be an $\mathrm{H}\Phi_k$-group and $i\colon H\to  G$ be an injective group homomorphism. Then  restriction along $i$ 
    \[
    \res^G_H\colon \Mod(kG)\to \Mod(kH)
    \]
   is both a right  Quillen functor and a left Quillen functor, with  left adjoint $\ind_H^G$  given by induction along $i$, and right adjoint $\mathrm{coind}_H^G$ given by coinduction along $i$. In particular, we obtain adjunctions of underlying $\infty$-categories 
    \[
    \begin{tikzcd}[column sep=large, row sep=large]
     \mathbf{StMod}(kH)\arrow[r,"\ind_H^G", yshift=3mm] \arrow[r, yshift=-3mm, "\CoInd_H^G"'] & \mathbf{StMod}(kG) .\arrow[l, "\res^G_H" description] 
    \end{tikzcd}
    \]
\end{Prop}

\begin{proof}
 First, we show that $\res^G_H$ is right Quillen. Recall that restriction is exact, and hence it preserves fibrations. It remains to verify that induction preserves cofibrations. Since $kG$ is projective as a $kH$-module, and $\ind_H^G = kG\otimes_{kH}-$, we obtain that induction is exact. Now, let $M$ be a Gorenstein projective $kH$-module which is the kernel in a totally acyclic complex $P_\ast$ of projective $kH$-modules. We claim that $\ind_H^G(P_\ast)$ is a totally acyclic complex of $kG$-modules. Note that $\ind_H^G(P_\ast)$ is an acyclic complex of projective $kG$-modules since induction is exact and preserves projective modules. By \cref{rem-totally-acy}, $\ind_H^G(P_\ast)$ is totally acyclic. We deduce that $\ind_H^G(M)$ is Gorenstein projective, as desired.

Now, let us show that $\res^G_H$ is left Quillen. Note that coinduction is exact since $\mathrm{coind}^G_H(-)=\mathrm{Hom}_{kH}(kG,-)$ and $kG$ is $kH$-projective. It follows that $\mathrm{coind}^G_H$ preserves fibrations. It remains to verify that $\res^G_H$ preserves cofibrations, but this is clear by \cref{rem-totally-acy}, since restriction is exact and preserves projectives.
\end{proof}

\begin{Prop}\label{prop-conservativity for stab}
     Let $G$ be a group of type $\mathrm{H}\Phi_k$ and let $\cat F_G$ be the family of finite subgroups of $G$. Then the family of functors 
     \[
     \{\mathrm{res}^G_H\colon \mathbf{StMod}(kG) \to\mathbf{StMod}(kH)\}_{H\in \cat F_G}
     \]
     is jointly conservative. 
\end{Prop}

\begin{proof}
    Let $\alpha$ be an ordinal such that $G\in \mathrm{H}_\alpha\Phi_k$. We argue by induction on $\alpha$. If $\alpha=0$, then the statement follows by \cite[Lemma 5.2]{mazza2019stable} (see \cref{rem-homotopycat of stab}). Now, assume that $\alpha>0$. Then there is a contractible finite-dimensional $G$-CW-complex $X$ such that the isotropy of each 0-dimensional cell is in $\mathrm{H}_\beta\Phi_k$ for some $\beta<\alpha$. Let $\mathcal{H}$ be the family of subgroups of $G$ which appear as the  isotropy group of some cell of $X$.   We claim that the family of functors  
    \[
     \{\mathrm{res}^G_H\colon \mathbf{StMod}(kG) \to \mathbf{StMod}(kH)\}_{H\in \mathcal{H}}
     \]
     is jointly conservative. Note that from this claim  the initial family is jointly-conservative as well by the inductive hypothesis. Let $M$ be a $kG$--module which is trivial in $\mathbf{StMod}(kH)$ for all $H\in \mathcal{H}$.  We want to show that $M\simeq 0$ in $\mathbf{StMod}(kG)$.   
    Recall that the augmented chain complex $\overline{C}_\ast(X)$ of $X$ is an exact sequence of $kG$--modules of the form   
   \begin{equation}\label{eq-cellular chain}
        0\to C_n\to C_{n-1} \to \ldots \to C_0\to k \to 0
   \end{equation}
    where each $C_i$ a coproduct of permutation $kG$--modules $k(G/H)$ with $H\in \mathcal{H}$. In particular, it is a split exact complex of $k$--modules (as it is an acyclic complex of $k$--projective modules). 
    Now,  note that tensoring the complex \eqref{eq-cellular chain} with $M$ and using the projection formula, gives us an exact complex of $kG$--modules where each term is a coproduct of modules of the form $\ind_H^G\res_H^G M$ for some $H\in \mathcal{H}$.   Since each short exact sequence in $\Mod(kG)$ induces a fiber sequence in $\mathbf{StMod}(kG)$, an inductive argument on the length of the complex \eqref{eq-cellular chain} gives us that $M$ is the cofiber in $\mathbf{StMod}(kG)$ of two trivial objects, hence it is trivial as well.
\end{proof}

We state the following criterion for showing that a stable $\infty$-category is compactly generated, which we will use repeatedly throughout the paper.

\begin{Prop}\label{prop-compactgen}
      Let $\cat I$ be a set, and let 
 \[
 \{f_i^\ast\colon \cat C \to \cat C(i)\}_{i\in \cat I}
 \]
be  a family of coproduct preserving triangulated functors between triangulated categories with small coproducts. Assume in addition that:
    \begin{enumerate}
        \item For any $i\in \cat I$, the functor $f_i^\ast$ admits a left adjoint $(f_i)_!$. 
       \item The category $\cat C(i)$ is compactly generated for each $i \in \cat I$.  
    \item The collection of functors $\{f^\ast_i\}_{i\in  \cat I}$ is jointly conservative. 
    \end{enumerate}
    Then $\cat C$ is compactly generated. 
\end{Prop}

\begin{proof}
      For each $i\in \cat I$, fix  a set of compact generators $\mathcal{X}_i$  of  $\cat C(i)$. Consider the set 
\[
\mathcal{X}=\{(f_i)_!(x)\mid x\in \mathcal{X}_i, \, i\in \cat I\}. 
\]
We claim that $\mathcal X$  is a set of compact generators of $\cat C$. Indeed, by our assumptions each $f_!(x)$ is compact in $\cat C$. Now,  note that an element $y\in \mathcal{X}^\perp$ is necessarily in 
\[
\bigcap_{i\in \cat I} \mathrm{Ker}(f^\ast_i)
\]
which is trivial by $(c)$. Hence  $y\simeq 0$ as desired. 
\end{proof}

As a first application of this criterion we deduce:
\begin{Cor}\label{cor-stmod-compactly-gen}
    Let $G$ be a group of type $\mathrm{H}\Phi_k$. Then $\mathbf{StMod}(kG)$ is compactly generated. 
\end{Cor}

\begin{proof}
    This follows from \cref{prop-conservativity for stab} combined with \cref{prop-compactgen}.
\end{proof}

\section{Homotopy categories of projective $kG$--modules}\label{sec-kproj}
In this section, we introduce the homotopy category of projective $kG$--modules, discuss when it gives rise to a compactly-generated tt-category, and finally characterize its compact objects.

\begin{Def}\label{def-kproj}
    Let $G$ be group and $k$ a commutative ring. We let $\mathbf{K}(\Mod(kG))$ denote the homotopy category of $kG$-modules which is a tensor-triangulated category with symmetric monoidal structure given by the tensor product over $k$. 
    We let $\mathbf{K}(\Proj(kG))$ denote the homotopy category of projective $kG$-modules which is a localizing tensor ideal in the category $\mathbf{K}(\Mod(kG))$.
\end{Def}
\begin{Rem}
For an arbitrary discrete group $G$, it is not clear whether the tensor product over $k$ endows $\mathbf{K}(\Proj(kG))$ with a monoidal structure, since a suitable monoidal unit may fail to exist (it is however a non-unital symmetric monoidal structure). We will address this issue, at least for groups of type $\Phi_k$ in \cref{cor-kproj-ttcat}. 
\end{Rem}

We now discuss the functoriality of the homotopy category of $kG$-modules.

\begin{Cons}\label{cons-ind-res-coind}
Given a subgroup $H \subseteq G$, we have an adjoint triple
  \[
    \begin{tikzcd}[column sep=large, row sep=large]
     \Mod(kH) \arrow[r,"\ind_H^G", yshift=3mm] \arrow[r, yshift=-3mm, "\CoInd_H^G"'] & \Mod(kG) \arrow[l, "\res^G_H" description] 
    \end{tikzcd}
    \]
of abelian categories. Passing to the associated homotopy categories of chain complexes we obtain another adjoint triple
 \[
    \begin{tikzcd}[column sep=large, row sep=large]
     \mathbf{K}(\Mod(kH)) \arrow[r,"\ind_H^G", yshift=3mm] \arrow[r, yshift=-3mm, "\CoInd_H^G"'] & \mathbf{K}(\Mod(kG)) \arrow[l, "\res^G_H" description] 
    \end{tikzcd}
    \]
which we denote in the same way. It is easy to see that the restriction functor is symmetric monoidal, and so the left and right adjoints are respectively oplax and lax monoidal. 
\end{Cons}

\begin{Prop}\label{prop-ind-res-kproj-mod}
    Let $H\subseteq G$ be a subgroup. The induction-restriction adjunction from \cref{cons-ind-res-coind}
    descends to an adjunction 
    \[
    \ind^G_H \colon \mathbf{K}(\Proj(kH))\rightleftarrows  \mathbf{K}(\Proj(kG))\colon \res_H^G. 
    \]
\end{Prop}

\begin{proof}
  We need to verify that the restriction and induction functors preserve projective objects. But this follows since $kG$ is projective as $kH$--module.   
\end{proof}

\begin{Prop}\label{prop-kproj-conservativity}
    Let $k$ be as in \cref{hyp} and $G$ be a group of type $\Phi_k$. 
     Then the family of functors 
\[
\{\res^G_H\colon  \mathbf{K}(\mathrm{Proj}(kG)) \to \mathbf{K}(\mathrm{Proj}(kH)) \mid H \subseteq G 
\; \mathrm{finite} \}
\]
is jointly conservative. 
\end{Prop}

\begin{proof}
    Let $M$ be in $\mathbf{K}(\mathrm{Proj}(kG))$ such that $\res^G_H(M)\simeq 0$ for all finite subgroups $H$ of $G$. In particular, the case $H=1$ implies that $M\in \mathbf{Ac}(\mathrm{Proj}(kG))$, the full subcategory of $\mathbf{K}(\mathrm{Proj}(kG))$ on  acyclic complexes. We observe that the restriction functor preserves acyclic complexes of projective modules, so there is a commutative diagram 
    \[
    \begin{tikzcd}
        \mathbf{Ac}(\mathrm{Proj}(kG)) \arrow[d,"\res^G_H"'] \arrow[r, hook] & \mathbf{K}(\Proj(kG)) \arrow[d,"\res^G_H"]\\
        \mathbf{Ac}(\mathrm{Proj}(kH)) \arrow[r, hook] & \mathbf{K}(\Proj(kH)). 
    \end{tikzcd}
    \]
    Therefore we are reduce to showing that $\{\res^G_H \mid H\subseteq G \; \mathrm{finite}\}$ is jointly conservative on categories of acyclic complexes. But by \cref{rem-ac=stmod} we can identify $\mathbf{Ac}(\mathrm{Proj}(kG))$ with the stable module category $\mathbf{StMod}(kG)$, and we already showed that the family of restriction functors is conservative for the stable module category; see \cref{prop-conservativity for stab}.
\end{proof}
We are ready to prove compact generation for the homotopy category of projective $kG$--modules.  
\begin{Prop}\label{prop-kproj-compl}
    Let $k$ be as in \cref{hyp} and $G$ be a group of type $\Phi_k$. Then  $\mathbf{K}(\mathrm{Proj}(kG))$ is compactly generated.  
\end{Prop}

\begin{proof}
In view of  \cref{prop-kproj-conservativity} and the fact that $\mathbf{K}(\mathrm{Proj}(kH))$ is compactly generated for any finite group $H$, we can invoke \cref{prop-compactgen} to deduce the result.
\end{proof}

As a consequence of compact generation of $\mathbf{K}(\mathrm{Proj}(kG))$, we deduce that it fits into a recollement generalizing the one in the finite case:

\begin{Prop}\label{recollement for kG}
   Let $k$ be as in \cref{hyp} and let $G$ be group of type $\Phi_k$. Then the canonical functor $\mathbf{q}\colon\mathbf{K}(\mathrm{Proj}(kG))\to \mathbf{D}(kG)$ induces  a recollement 
    \begin{center}
\begin{tikzcd}
\mathbf{StMod}(kG) \arrow[tail]{rr} &&  \mathbf{K}(\mathrm{Proj}(kG))
\arrow[twoheadrightarrow,yshift=1.5ex]{ll}[swap]{\mathbf{t}}
\arrow[twoheadrightarrow,yshift=-1.5ex]{ll}
  \arrow[twoheadrightarrow]{rr} && \mathbf{D}(kG)
   \arrow[tail,yshift=1.5ex]{ll}[swap]{\mathbf{p} }
   \arrow[tail,yshift=-1.5ex]{ll}{}
\end{tikzcd}    
\end{center}
where the functor  $\mathbf{p}$ is induced by taking $K$--projective resolutions.
\end{Prop}

\begin{proof}
As discussed in \cite[Section 4.3]{krause}, there is a localization sequence 
\[
\begin{tikzcd}
   \mathbf{Ac}(\mathrm{Proj}(kG)) \arrow[r,twoheadleftarrow,shift left=0.4ex] \arrow[r,draw=none]& 
    \mathbf{K}(\mathrm{Proj}(kG)) \arrow[l,leftarrowtail,shift left=0.75ex]\arrow[r,leftarrowtail,shift left=0.4ex,"\mathbf{p}"]& 
    \mathbf{D}(kG) \arrow[l,twoheadleftarrow,shift left=0.75ex,"\mathbf{q}"]
\end{tikzcd}    
\]
where the left hand side denotes the full subcategory of $\mathbf{K}(\mathrm{Proj}(kG))$ on acyclic complexes, the functor  $\mathbf{q}$ inverts quasi-isomorphisms, and the functor $\mathbf{p}$ takes $K$-projective resolutions. This localization can be promoted to the claimed recollement. Indeed, since  $ \mathbf{K}(\mathrm{Proj}(kG))$ is compactly generated  (see  \cref{prop-kproj-compl}) and $\mathbf{q}$ commutes with coproducts, we can invoke Brown's representability to obtain the right hand side of the recollement.  The other half is an immediate consequence. 
Moreover,  we can identify $\mathbf{Ac}(\mathrm{Proj}(kG))$ with the stable module category $\mathbf{StMod}(kG)$; see \cref{rem-ac=stmod}. 
\end{proof}

\begin{Prop}\label{prop-res-motimesn}
     Let $k$ be as in \cref{hyp} and $G$ be a group of type $\Phi_k$. For any subgroup $H\subseteq G$ and objects $X, Y \in \mathbf{K}(\Proj(kG))$ we have that $\res^G_H(X\otimes Y)\simeq \res^G_H X\otimes \res^G_HY$. 
\end{Prop}

\begin{proof}
     This follows since $\mathbf{K}(\Proj(kG))$ is closed under the tensor product of the category $\mathbf{K}(\Mod(kG))$ and the functor $\res^G_H$ is induced by the symmetric monoidal  functor 
    \[
    \res^G_H\colon \mathbf{K}(\Mod(kG)) \to \mathbf{K}(\Mod(kH))
    \] 
    see \cref{cons-ind-res-coind}. 
\end{proof}

\subsection{A characterization of compact objects}

The goal of this subsection is to give a characterization of the compact objects in $\mathbf{K}(\mathrm{Proj}(kG))$. We need some preparation. 

\begin{Def}
    We let $\mathrm{Gp}(kG)$ denote the full subcategory of  $\Mod(kG)$ on modules which are isomorphic to a kernel in a totally acyclic complex of finitely generated $kG$--modules. 
\end{Def}

\begin{Rem}\label{rem-Gp-contained-GP}
    Note that $\mathrm{Gp}(kG)$ is an extension closed subcategory of $\mathrm{GP}(kG)$ by the horseshoe lemma. Moreover, it has enough projectives by definition; see \cref{rec-GP-Frobenius}. Since the category $\mathrm{GP}(kG)$ is Frobenius exact (see \cref{rec-GP-Frobenius}), we deduce that  $\mathrm{Gp}(kG)$ is a Frobenius exact  category where the projective-injective objects correspond to the finitely generated projective $kG$--modules. 
\end{Rem}

We also adopt some notation from \cite[Section 2]{BBIKP}. 

\begin{Not}
    Let $\mathcal{C}$ be a full additive subcategory of $\mathrm{GP}(kG)$, and let $\mathbf{K}(\mathcal{C})$ denote the homotopy category. We write $\mathbf{K}^+(\mathcal{C})$ to denote the full subcategory of $\mathbf{K}(\mathcal{C})$ on complexes $X$ such that $X^n=0$ for all $n\ll0$, and 
    \[
    \mathbf{K}^{+,b}(\mathcal{C})=\{X\in \mathbf{K}^{+}(\mathcal{C})\mid H^n(X)=0 \mbox{ for } |n|\gg0\}
    \]
    where $H^n(X)=0$ means that the differential $d^{n-1}_X$ factors as 
    \[
    X^{n-1} \twoheadrightarrow \mathrm{ker}(d^n_X)\rightarrowtail X^n
    \]
   the composite of admissible epimorphism and an admissible monomorphism in the exact category $\mathrm{GP}(kG)$.
\end{Not}

\begin{Not}
    We write $\mathrm{proj}(kG)$ to denote the full subcategory of $\Mod(kG)$ on finitely generated projective $kG$--modules.
\end{Not}

\begin{Prop}\label{prop-identification-Gproj}
     Let $k$ be as in \cref{hyp} and $G$ be a group. Then the inclusion functor $\mathrm{proj}(kG)\to \mathrm{Gp}(kG)$ induces an equivalence of triangulated categories
     \[
     \mathbf{K}^{+,b}(\mathrm{proj}(kG)) \to \mathbf{D}^b(\mathrm{Gp}(kG)).
     \]
\end{Prop}

\begin{proof}
    This follows from the fact that $\mathrm{Gp}(kG)$ is a Frobenius exact category combined with \cite[Corollary 4.2.9]{krause}. 
\end{proof}

 We need to recall the following result.

\begin{Rec}\label{rec-compacts-in-loc}
    Let $\cat T$ be a compactly generated triangulated category. For a collection of compact objects $\mathcal{X}$, we have the following equality 
    \[
    \mathrm{loc}(\mathcal{X})\cap \cat T^c =\mathrm{thick}(\mathcal{X})
    \]
    see \cite[Lemma 2.2]{Neeman92}. 
\end{Rec}

\begin{Lem}\label{lemma-compacts}
     Let $k$ be as in \cref{hyp} and $G$ be a group of type $\Phi_k$.  Then $\mathbf{K}(\Proj(kG))^c$ identifies with $\mathbf{K}^{+,b}(\mathrm{proj}(kG))$. Moreover, we have 
     \[
     \mathbf{K}(\Proj(kG))^c\simeq \mathbf{D}^{b}(\mathrm{Gp}(kG)).
     \]
\end{Lem}

\begin{proof}
   Note that the second claim is a direct consequence of the first claim by \cref{prop-identification-Gproj}. Now,  in \cref{prop-kproj-compl}, we described a set of compact generators of $\mathbf{K}(\Proj(kG))$, and in view of \cref{rec-compacts-in-loc} it follows that 
    \[
    \mathbf{K}(\Proj(kG))^c=\mathrm{thick}(\ind_H^G X\mid X\in \mathbf{K}(\Proj(kH))^c, \, H\subseteq  G \mbox{ finite} ). 
    \]
    On the other hand, for a finite group $H$, we have 
    \[
    \mathbf{K}(\Proj(kH))^c=\mathbf{K}^{+,b}(\mathrm{proj}(kH))
    \]
    as proved in \cite[Lemma 2.13]{BBIKP}. Since $\ind_H^G$ is exact and sends finitely generated projectives to finitely generated projectives, we obtain 
     \[
    \mathbf{K}(\Proj(kG))^c\subseteq  \mathbf{K}^{+,b}(\mathrm{proj}(kG)). 
    \]
    For the other inclusion, we use the equivalence from \cref{prop-identification-Gproj}. By a thick subcategory argument, we can assume that $X\in \mathbf{K}^{+,b}(\mathrm{proj}(kG))$ is of the form $\mathbf{i}M$ for some $M\in \mathrm{Gp}(kG)$. Here $\mathbf{i}M$ denotes an injective resolution of $M$ in the Frobenius exact category $\mathrm{Gp}(kG)$ which is the inverse of the equivalence $ \mathbf{K}^{+,b}(\mathrm{proj}(kG)) \to \mathbf{D}^b(\mathrm{Gp}(kG))$ from \cref{prop-identification-Gproj}. In this case, we have that for any $Y\in \mathbf{K}(\mathrm{Proj}(kG))$, the map $M\to \mathbf{i}M$ induces a bijection
    \[
    \Hom_{\mathbf{K}(\mathrm{Proj}(kG))}(\mathbf{i}M, Y)\simeq \Hom_{\mathbf{K}(\mathrm{GP}(kG))}(M, Y)
    \]
    see \cite[Lemma 4.2.6]{krause}. Since $M$ lies in $\mathrm{Gp}(kG)$, it is in particular finitely presented, and hence  $\Hom_{\mathbf{K}(\mathrm{GP}(kG))}(M, -)$ preserves filtered colimits. It follows that $\mathbf{i}M$ is compact, and this completes the proof. 
\end{proof}

\begin{Rem}
    Let $\mathrm{GP}^{\mathrm{fg}}(kG)$ denote the full subcategory of $\mathrm{GP}(kG)$ on finitely generated modules. If follows that $\mathrm{Gp}(kG)$ is a subcategory of $\mathrm{GP}^{\mathrm{fg}}(kG)$. Moreover, these two categories agree when the group algebra is coherent; see \cite[Proposition 1.4]{Zhang2008ABI}. This is the case for instance when $G$ acts on a tree with finite isotropy groups; see \cite[Proposition 3.3]{CEKT14}.
\end{Rem}

We have a further characterization of the category $\mathbf{D}^b(\mathrm{Gp}(kG))$. For this, let us consider the category $\mod_\infty(kG)$ of  $kG$--modules of type $FP_\infty$; that is, $kG$--modules that admit a resolution by finitely generated projective $kG$--modules. By the horseshoe lemma this category is an extension-closed subcategory of $\Mod(kG)$, hence it inherits an exact structure and has enough projectives.  Moreover, this category satisfies the 2-out-of-3 property. That is, for any short exact sequence of $kG$--modules, if any two out of the three modules are of type $FP_\infty$, then so is the third. See \cite[Proposition 1.4]{Bie81}.

\begin{Prop}
\label{prop-gp-equivalent-mod}
  Let $k$ be as in \cref{hyp} and $G$ be a group of type $\Phi_k$.    Then the inclusion $\mathrm{Gp}(kG)\subseteq \mod_\infty(kG)$ induces an equivalence of triangulated categories 
    \[
    \mathbf{D}^b(\mathrm{Gp}(kG))\simeq \mathbf{D}^b(\mathrm{mod}_\infty(kG)).
    \]
\end{Prop}
\begin{proof}
     In view of \cite[Theorem 12.1]{Kel96}  (see also \cite[Theorem 1.1(2)]{henrard2020preresolving}), we only need to verify that $\mathrm{Gp}(kG)$ is a finitely preresolving subcategory of $\mathrm{mod}_\infty(kG)$, which amounts to prove: 
     \begin{itemize}
         \item $\mathrm{Gp}(kG)$ is closed under kernels of admissible epimorphisms in $\mod_\infty(kG)$. Let $0\to L\to M \to N\to 0$ a short exact sequence of $kG$--modules with $M$ and $N$ in $\mathrm{Gp}(kG)$. Since both $\mathrm{GP}(kG)$ and $\mod_\infty(kG)$ satisfy the 2-out-of-3 property, we conclude that $L$ must be a Gorenstein  projective $kG$--module of type $FP_\infty$. From \cite[Corollary 5.3]{Ben97}, we deduce that $L$ has a complete resolution by finitely generated projective $kG$--modules, showing that $L \in \mathrm{Gp}(kG)$. Here we are using that the class of Gorenstein projective modules agrees with the class of Benson's cofibrant modules. See \cref{rem-bensoncof}. 
         \item For any $M$ in $\mod_\infty(kG)$, there is an admissible epimorphism $N\to M$ with $N\in\mathrm{Gp}(kG)$. This follows since finitely generated projective modules  lie in $\mathrm{Gp}(kG)$. 
         \item Any $M$ in  $\mod_\infty(kG)$ has a finite resolution by modules in $\mathrm{Gp}(kG)$. This is a consequence of \cite[Corollary 5.3]{Ben97} using again \cref{rem-bensoncof}. Indeed, the result cited gives us that a high enough syzygy $\Omega^n M$ of $P^\bullet$, a resolution of $M$ by finitely generated projective $kG$--modules, must have a complete resolution in the sense of \cite[Definition 5.1]{Ben97}, equivalently, such syzygy must be Benson's cofibrant, and so a Gorenstein projective module.  It follows that  $0\to \Omega^n M \to P^n\to P^{n+1}\to P^0\to M\to 0$ is a finite resolution of $M$ by modules in $\mathrm{Gp}(kG)$.
     \end{itemize}
     This completes the proof. 
\end{proof}

Let $\mathrm{fp}(kG)$ denote the category of finitely presented $kG$--modules. 

\begin{Cor}\label{cor-compact-Kproj}
     Let $k$ be as in \cref{hyp} and $G$ be a group of type $\Phi_k$.  Then we have an equivalence of triangulated categories  
     \[
     \mathbf{K}(\Proj(kG))^c\simeq \mathbf{D}^{b}(\mathrm{mod}_\infty(kG)).
     \]
     Moreover, if $kG$ is coherent (for instance, if $G$ has a 1-dimensional model for $\underline{E}G$), then there is an equivalence of triangulated categories 
     \[
     \mathbf{K}(\Proj(kG))^c\simeq \mathbf{D}^{b}(\mathrm{fp}(kG)).
     \]
\end{Cor}

\begin{proof}
    This is consequence of \cref{lemma-compacts} together with \cref{prop-gp-equivalent-mod}. The last claim follows since in this case $\mod(kG)$ agrees with $\mathrm{fp}(kG)$.  
\end{proof}

\section{Derived categories of permutation modules}\label{sec-perm}

In this subsection, let $G$ be a group admitting a finite-dimensional model for the classifying space for proper actions, and let $k$ be a commutative Noetherian ring. Our goal is to introduce a \textit{derived category of permutation $kG$–modules}, in analogy with the construction of Balmer and Gallauer in \cite{balmer2023finite}. This category was also considered in \cite{kendall2024co} in the case where $k$ is a field of positive characteristic.

\begin{Not}
    Let $\mathrm{perm}(G,k)$ denote the additive category of permutation $kG$--modules; in symbols: 
    \[
    \mathrm{perm}(G,k)\coloneqq \mathrm{add}\langle k(G/H)\mid H\subseteq G\rangle.
    \]
     We write $\mathrm{perm}^{\fin}(G,k)$ to denote the additive closure, in $\mathrm{perm}(G,k)$, of the collection of \textit{finite permutation modules}, that is,  modules of the form $k(G/H)$ for a finite subgroup $H\subseteq G$.  
\end{Not}

\begin{Rem}
   In general, the category $\mathrm{perm}(G,k)$ is not idempotent complete, even for finite groups; see \cite{BG23}.
\end{Rem}

\begin{Conv}
   We refer to the modules in the idempotent completion of $\mathrm{perm}(G,k)$ as $\natural$-permutation modules, following \cite{balmer2023finite}. Similarly, we use the term \textit{finite $\natural$-permutation modules} for objects of $\mathrm{perm}^\mathrm{fin}(G,k)^\natural$.
\end{Conv}

\begin{Def}
     The \textit{derived category of finite permutation $kG$--modules} $$\DPerm^{\mathrm{fin}}(G,k)$$ is defined as the localizing closure, in the homotopy category $\mathbf{K}(\Mod(kG))$, of $\mathrm{perm}^{\fin}(G,k)$; in symbols, 
     \[
     \DPerm^{\fin}(G,k)\coloneqq \Loc_{\mathbf{K}(\Mod(kG))}(\mathrm{perm}^{\fin}(G,k)).
     \]
\end{Def}

\begin{Rem}
    If $G$ is a finite group, then $\DPerm^{\fin}(G,k)$ agrees with the Balmer-Gallauer's derived category of permutation modules; see \cite[Proposition 3.9]{BG23}. 
\end{Rem}

\begin{Rem}
    In general, the categories $\DPerm^{\fin}(G,k)$ and $\mathbf{K}(\mathrm{Perm}^{\fin}(G,k))$ do not agree, where $\mathrm{Perm}^{\fin}(G,k)$ denotes the additive subcategory of the category of all permutation modules generated by the finite permutation modules;  see \cite[Remark 3.5]{BG23}. 
\end{Rem}

\begin{Lem}\label{rem-dperm is comp}
    The triangulated category $\DPerm^{\fin}(G,k)$ is compactly generated by the permutation modules $\mathrm{perm}^{\fin}(G,k)$. In fact, we have 
    \[
    \DPerm^{\fin}(G,k)^c=\mathbf{K}^b(\mathrm{perm}^{\fin}(G,k))^\natural=\mathbf{K}^b(\mathrm{perm}^{\fin}(G,k)^\natural).
    \]
\end{Lem}
\begin{proof}
     By definition of the derived category of permutation modules, the set $\mathrm{perm}^{\fin}(G,k)$ generates so we just need to verify that the objects are compact.  Now recall from \cref{cons-ind-res-coind} that the restriction functor $\mathbf{K}(\Mod(kG))\to \mathbf{K}(\Mod(kH))$ admits a left and right adjoint. In particular, the left adjoint (i.e., induction) preserves compact objects. Finally, recall that $k\in \DPerm^{\fin}(H,k)$ is compact for any finite group $H$. Combining these facts, we see that 
    \[
    k(G/H)\simeq \mathrm{ind}_H^G(k) \in \DPerm^{\fin}(G,k)^c
    \]
    for any finite subgroup $H\subseteq G$. Using \cref{rec-compacts-in-loc} we deduce that compact objects agrees with the thick subcategory generated by the finite permutation modules, which in turn can be identify with $\mathbf{K}^b(\mathrm{perm}^{\fin}(G,k))^\natural$ via a thick subcategory argument. The last equality of the lemma follows by noting that retracts of complexes are degreewise.  
\end{proof}

\begin{Prop}\label{prop-ind-res-permutation-mod}
    Let $H\subseteq G$ be a finite subgroup. The induction-restriction adjunction from \cref{cons-ind-res-coind}
    descends to an adjunction 
    \[
    \ind^G_H \colon \DPerm^{\fin}(H,k) \leftrightarrows  \DPerm^{\fin}(G,k)\colon \res_H^G. 
    \]
     Moreover, the restriction functor $\res^G_H$ has also a right adjoint.
\end{Prop}

\begin{proof}
  For the first claim we only need to verify that the restriction and induction functors take permutation modules to permutation modules. But this follows by an application of Mackey formula and the fact that 
  \[
  \ind_H^G(k(H/L))\simeq \ind_L^G(k)\simeq k(G/L)
  \]
  for any subgroup $L\subseteq H$.  The second claim is a consequence of Brown's representability and  \cref{rem-dperm is comp}.  
\end{proof}

\begin{Rem}\label{rem:conservativity}
    Note that the family of functors 
    \[
    \{\res^G_H\colon  \DPerm^{\fin}(G,k) \to \DPerm^{\fin}(H,k) \mid H \subseteq G \;\;\text{finite}\}
    \]
    is jointly-conservative. This is a consequence of the definition of $\DPerm^{\fin}(G,k)$. Indeed, an element $M\in \DPerm^{\fin}(G,k)$ such that $\res^G_H(M)\simeq 0$ for all finite subgroups $H\subseteq G$ must be lie in 
    \[
    \DPerm^{\fin}(G,k)\supseteq\mathrm{perm}^{\fin}(G,k)^\perp=0
    \]
    by the adjunction induction-restriction.
\end{Rem}

 \begin{Rem}
        For a group $G$ with a finite-dimensional model $X$ for $\underline{E}G$, it is also natural to consider  the category 
             \[
     \DPerm(G,k)\coloneqq \Loc_{\mathbf{K}(\Mod(kG))}(\mathrm{perm}(G,k))
     \]
   of all permutation modules and not just those arising from finite subgroups. If our group $G$ is infinite, we claim that  the restriction functors
     \[
     \res^G_H\colon  \DPerm(G,k) \to  \DPerm(H,k)
     \]
     for all finite subgroups of $H\subseteq G$ do not determine a family of jointly-conservative functors. Indeed, assume that the family of functors $\{\res^G_H\}_{H\in \Fin(G)}$ is jointly conservative. Then the augmentation map 
     \[
     C_\ast(X) \to k 
     \]
     is an isomorphisms up to restriction to any finite subgroup $H$ since the space $X$ is $H$-contractible; see the proof of \cref{prop-unit in Dperm}. It follows that the augmented complex $\overline{C}_\ast(X)$ is a contractible complex of $kG$--modules. In particular, we must have that for any $H\subseteq G$, the complex $\overline{C}_\ast(X)^H$ is contractible. But since $G$ is infinite and the isotropies of the permutation modules appearing in $C_\ast(X)$ are finite, we get that $\overline{C}_i(X)^G=0$ for all $i\neq 0$, and $\overline{C}_0(X)^G=k$. Hence $\overline{C}_\ast(X)$ cannot be $G$-contractible.

As a consequence of this observation, we get that $\DPerm^{\fin}(G,k)$ is a proper subcategory of $\DPerm(G,k)$.
    \end{Rem}

    \begin{Rem}
    Recall that the abelian category $\Mod(kG)$ of $kG$--modules has a symmetric monoidal structure given by the tensor product over $k$ equipped with the diagonal $G$--action. In particular, tensoring two permutation modules is again a permutation module by the Mackey formula. As a consequence, we obtain that $\DPerm^{\fin}(G,k)$ is stable under the tensor product of $\mathbf{K}(\Mod(kG))$. In fact, we will see later that it is in fact a tensor triangulated category. Note that the missing part is the monoidal unit.
    \end{Rem} 

Recall that $k$ denotes a commutative Noetherian ring. 

\begin{Prop}\label{prop-unit in Dperm}
    The derived category of permutation modules $\DPerm^{\fin}(G,k)$ is a tensor triangulated category with the tensor structure inherited form the one in $\mathbf{K}(\Mod(kG))$; the monoidal unit is given by the cellular chain complex $C_\ast(X)$ for any finite-dimensional model $X$ for $\underline{E}G$. 
\end{Prop}

\begin{proof}
 We already know that $\DPerm^{\fin}(G,k)$ is closed under the tensor product, it remains to verify the claim about the tensor unit.  Fix a finite subgroup $H\subseteq G$. By considering $X$ as a $H$-space, we claim that the augmented cellular chain complex  $\overline{C}_\ast(X)$  is $H$--contractible. Indeed, since $X^K$ is contractible for each $K\subseteq H$, by Whitehead's theorem we deduce that $X$ is $H$--contractible as desired. In other words, we get that the augmentation map  $C_\ast(X)\to k$ is a homotopy equivalence of complexes of $kH$--modules, and hence an isomorphism in $\DPerm^{\fin}(H,k)$ since $k$ is itself is in $\mathrm{perm}^{\fin}(H)$. Now, let $M$ be in $\DPerm^{\fin}(G,k)$. We need to show that $M\otimes C_\ast(X)\simeq M$. But note that the map   
\[
\mathrm{id}_M\otimes \epsilon \colon M\otimes C_\ast(X)\to M\otimes k\simeq M
\]
obtained by tensoring with the augmentation map is a homotopy equivalence after restriction to any finite subgroup  $H\subseteq G$ by the previous claim. Since the map $\mathrm{id}_M\otimes \epsilon$ is in $\DPerm^{\fin}(G,k)$, and $\res^G_H(\mathrm{id}_M\otimes \epsilon)$ is an isomorphism in $\DPerm^{\fin}(H,k)$ for any finite subgroup $H\subseteq G$, we conclude that $\mathrm{id}_M\otimes \epsilon$ is also an isomorphism in $\DPerm^{\fin}(G,k)$ by  \cref{rem:conservativity}.     
\end{proof}

\begin{Rem}
   Combining \cref{prop-unit in Dperm} and \cref{rem-dperm is comp} we see that the unit object in $ \DPerm^{\fin}(G,k)$ is compact if $G$ admits a finite model for $\underline{E}G$. 
\end{Rem}

\begin{Cor}
    Let $H\subseteq G$ be a finite subgroup. Then the restriction functor 
    \[
    \res^G_H\colon \DPerm^{\fin}(G,k) \to \DPerm^{\fin}(H,k)
    \]
    is symmetric monoidal. 
\end{Cor}

\begin{proof}
    Recall that  the monoidal structure of $\DPerm^{\fin}(G,k)$ is induced from the one in $\mathbf{K}(\Mod(kG))$ and the restriction functor $\mathbf{K}(\Mod(kG))\to \mathbf{K}(\Mod(kH))$ is symmetric  monoidal. Hence, we obtain that 
    \[
    \res^G_H(M\otimes N)\simeq \res^G_H(M)\otimes \res^G_H(N)
    \]
    in $\DPerm^{\fin}(H,k)$ for any $M$ and $N$ in $\DPerm^{\fin}(G,k)$. It remains to verify that $\res^G_H(C_\ast(X))\simeq k$. But we already showed this in the proof of the previous proposition. 
\end{proof}

\section{Functoriality}\label{sec-functoriality}

In the previous sections we defined three categories of representations
\[
\mathbf{StMod}(kG) \quad \mathbf{K}(\Proj(kG)) \quad \mathrm{and} \quad \DPerm^{\fin}(G,k)
\]
and in this section we discuss the functoriality of these categories on the group variable. Let us first introduce two important indexing categories and explain the relation between them.

\begin{Def}\label{def-G-Orb-proper}
    Let $G$ be a discrete group. We write $\Orb_G$ for the $G$-orbit $\infty$-category whose objects are the cosets $G/H$ for $H$ a subgroup of $G$, and whose morphisms are the $G$-equivariant maps. If $\cat F$ is a collection of subgroups of $G$, we write $\Orb_{G,\cat F}$ for the full subcategory of the orbit $G$-category on objects with isotropy in $\cat F$. If $\cat F$ consists of the finite subgroups of $G$, then we write $\Orb_{G,\fin}$ rather than $\Orb_{G, \cat F}$. 
\end{Def}
It will be helpful to have a model of the orbit category as a $2$-category. To this end, recall from \cite[Theorem 00AC]{kerodon} that the Duskin Nerve gives us a way to view $2$-categories as $\infty$-categories, in the same way the usual nerve construction allows us to view $1$-categories as $\infty$-categories.  
\begin{Rem}\label{rem-orbit-2-category}
    Consider the following strict $(2,1)$-category $\mathbf{Orb}_{G}$:
    \begin{itemize}
    \item objects are cosets $G/H$ with $H\subseteq G$ a subgroup;
    \item $1$-morphisms $G/H \to G/K$ are given by $g \in G$ such that $c_g(H)\subseteq K$;
    \item $2$-morphisms $g \Rightarrow g'$ are given by $k \in K$ such that $kg=g'$.
    \end{itemize}
    The Duskin Nerve of $\mathbf{Orb}_{G}$ is a model for the $\infty$-category $\Orb_{G}$. Similar statements hold if we restrict the isotropy.
\end{Rem}
 \begin{Def}
     We define the $\infty$-category $\Orb$ as the Duskin Nerve of the strict $(2,1)$-category $\mathbf{Orb}$ of groups (not necessarily finite), injective groups homomorphisms and conjugations.
 \end{Def}
\begin{Cons}\label{lem-can-functor}
    Let $G$ be a discrete group. We define a strict $2$-functor $\mathbf{Orb}_{G} \to \mathbf{Orb}$ as follows:
    \begin{itemize}
        \item On objects, the functor sends the coset $G/H$ to $H$.
        \item On $1$-morphisms, the functor sends an element $g \in G$ such that $c_g(H)\subseteq K$ to the inner automorphism $c_g\colon H \to K$.
        \item On $2$-morphisms, the functor sends an element $k\in K$ such that $kg=g'$  to the inner automorphism $c_k \colon c_g \Rightarrow c_{g'}$.
    \end{itemize}
    Applying the Duskin Nerve to the above functor and using \cref{rem-orbit-2-category} we obtain a functor of $\infty$-categories
    \[
    \kappa \colon  \Orb_G \to \Orb.
    \]
    If we restrict the isotropy to a collection of subgroups $\cat F$ we find a functor 
   \[
  \kappa \colon \Orb_{G,\cat F} \to \Orb.   
   \]
\end{Cons}

\begin{Cons}\label{cons-funct-Kproj-KMod}
The starting point is to note that there is a functor 
 \[
 \Mod(k -) \colon \Orb^{\opname} \to \Cat^\otimes, 
 \]
 sending a group $G$ to the symmetric monoidal abelian category $(\Mod(kG), \otimes_k)$, an injective group homomorphism $ i\colon H \to G$ to the restriction functor 
 \[
 \res^G_H\colon \Mod(kG) \to \Mod(kH)
 \]
 and a conjugation $c_g\colon i \to j$ between two inclusions $i,j \colon H \to G$, so $c_g \circ i =j$, to the transformation 
 \[
i^* l_g \colon j^*=i^*c_g^* \Rightarrow i^*
 \]
 where $l_g\colon c_g^* \Rightarrow \mathrm{id}_G$ is the natural transformation induced by the left multiplication by $g$.
 
 Passing to chain complexes and applying the dg-nerve give rise to a functor 
 \[
 \mathbf{K}(\Mod(k -)) \colon \Orb^{\opname} \to\widehat{\Cat}^\otimes_\infty
 \]
 into symmetric monoidal $\infty$-categories.
 We can restrict this functor further to $\Orb_G$ or $\Orb_{G,\cat F}$ for a discrete group $G$ using the maps described in \cref{lem-can-functor}.
\end{Cons}
Recall the notation $\CAlg(\PrL)$ from the convention section.  
Inverting quasi-isomorphisms pointwise gives us the following:
\begin{Cor}\label{cor-funct-der}
   Let $k$ be a commutative ring and $G$ be a discrete group.  Then the functor from \cref{cons-funct-Kproj-KMod} induces a functor on derived categories
    \[
 \mathbf{D}(k -) \colon \Orb_{G}^{\opname}\to\CAlg(\PrL).
 \]
\end{Cor}

\begin{Cor}\label{cor-functor-kproj}
    Let $k$ be as in \cref{hyp} and $G$ be a group of type $\Phi_k$.
    Then the functor from \cref{cons-funct-Kproj-KMod} restricts to a functor 
    \[
 \mathbf{K}(\Proj(k -)) \colon \Orb_{G}^{\opname}\to\PrL. 
 \]
\end{Cor}
\begin{proof}
    It suffices to notice that the restriction functors preserves projective modules, so the functor from \cref{cons-funct-Kproj-KMod} restricts to $ \mathbf{K}(\Proj(k -))$. The fact that the functor lands in $\PrL$ then follows from \cref{prop-kproj-compl} and the fact that the restriction functors preserves all colimits.
\end{proof}
We will see in  \cref{cor-kproj-ttcat} that the above functor lifts to $\CAlg(\PrL)$.

\begin{Cor}\label{cor-funct-dperm}
     Let $k$ be as in \cref{hyp} and let $G$ be a group with finite-dimensional model for $\underline{E}G$. Then the functor from \cref{cons-funct-Kproj-KMod} restricts to a functor 
     \[
 \DPerm^{\fin}(-,k) \colon \Orb_{G}^{\opname} \to\CAlg(\PrL). 
 \]
\end{Cor}
\begin{proof}
    The restriction functors preserves permutation modules, so the functor from \cref{cons-funct-Kproj-KMod} restricts to $ \DPerm^{\fin}(-,k)$. The fact that the functor lands in the correct category follows by combining \cref{rem-dperm is comp}, \cref{prop-ind-res-permutation-mod} and \cref{prop-unit in Dperm}.
\end{proof}

\begin{Cor}\label{cor-funct-stmod}
     Let $k$ be as in \cref{hyp} and let $G$ be a group of type $\mathrm{H}\Phi_k$. Then the stable module category defines a functor 
     \[
 \mathbf{StMod}(k-) \colon \Orb_{G}^{\opname} \to\CAlg(\PrL). 
 \]
\end{Cor}
\begin{proof}
    By \cref{rem-ac=stmod} we can identify the stable module category with the subcategory of acyclic complexes of projectives. Therefore the required functor is obtained from \cref{cons-funct-Kproj-KMod} by noting that the restriction functor preserves acyclic complexes of projectives since it is exact.  The fact that the functor lands in the correct category follows by combining \cref{rem-stmod-sym-mon}, \cref{prop-induction-stmod} and \cref{cor-stmod-compactly-gen}.
\end{proof}

\section{Descent}\label{sec-descent}
The goal of this section is to prove a descent result for special types of functors out of the orbit $G$-category.  Let us first discuss the type of functors we are interested in. 

\begin{Not}\label{nota-C-orbG}
For a group $G$, we write $\cat S_{G}$ for the $\infty$-category of $G$-spaces which is defined as the presheaf category on $\Orb_G$. Then any functor $\cat C \colon \Orb_G^{\opname}\to \PrL$ admits an unique limit-preserving extension to a functor $\cat S_{G}^{\opname} \to \PrL$, which by a slight abuse of notation we still denote by $\cat C$. To simplify the notation, we will write $\cat C(H)$ for the value of the functor $\cat C$ at the orbit $G/H$.
\end{Not}
Recall from \cref{conventions} that a collection of subgroups of a group $G$ is assumed to be closed under $G$-conjugation. 
\begin{Hyp}\label{hyp-functor}
    Let $\cat F$ be a collection of subgroups of a group $G$, and let $\cat C$ be a functor as in \cref{nota-C-orbG}. Assume in addition that:
    \begin{enumerate}
        \item For any $G$-equivariant map $f \colon G/K \to G/H$, the corresponding functor $f^* \colon \cat C(H) \to \cat C(K)$ admits a left adjoint $f_!$. When $f$ is induced by an inclusion we will write $f^*=\res^H_K$ and $f_!=\ind_K^H$.
        \item For any pullback square in $\cat S_{G}$
    \[
    \begin{tikzcd}
        X \arrow[r] \arrow[d] & G/L\arrow[d] \\
        G/H \arrow[r] & G/K
    \end{tikzcd}
    \]
    the corresponding diagram 
    \[
    \begin{tikzcd}
        \cat C(K)\arrow[r] \arrow[d] & \cat C(H)\arrow[d] \\
        \cat C(L) \arrow[r] & \cat C(X)
    \end{tikzcd}
    \]
    is horizontally left adjointable, see \cite[Definition 4.7.4.13]{HA}. 
    \item The collection of functors $\{\res^G_H\}_{H \in \cat F}$ is jointly conservative. 
    \end{enumerate}
\end{Hyp}

\begin{Rem}\label{rem-G-presentable}
Using a language more familiar to representation theorists, condition (b) above is requiring that restriction and induction functors satisfy the Mackey formula, also known as the double coset formula. On the other hand using the language of parametrized homotopy theory, we can compactly phrase conditions (a) and (b) as stating that the functor $\cat C$ is $G$-presentable, see \cite[Definition 4.3 and Lemma 4.9]{CLL23}. In particular comparing with \cite[Definition 2.16]{CLL23} we see that (a) and (b) imply that:
\begin{itemize}
    \item[(a')] for any map of $G$-spaces $f \colon X \to Y$, the induced functor $f^*$ admits a left adjoint $f_!$;
    \item[(b')] pullback squares in $G$-spaces (and not just those arising from orbits) give rise to a left adjointable squares. 
\end{itemize}
\end{Rem}

The goal of this section is to prove the following descent result. 

\begin{Thm}\label{thm-descent}
    In the situation of \cref{hyp-functor}, the restriction functors induce an equivalence of stable $\infty$-categories
    \[
    \cat C(G) \xrightarrow{\sim}\lim_{G/H \in \Orb_{G,\cat F}^{\opname}} \cat C(H).
    \]
    Moreover if the functor $\cat C$ lifts to $\CAlg(\PrL)$, then the above equivalence is symmetric monoidal.
\end{Thm}

We will split this result in several steps. 
\begin{Lem}\label{lem1}
     Write $\Set_{G, \cat F}$ for the coproduct completion of $\Orb_{G,\cat F}$ and consider the canonical functor $i \colon \Orb_{G,\cat F}^{\opname} \to \Set_{G,\cat F}^{\opname}$. Restriction along $i$ induces an equivalence
    \[
    \lim_{X \in \Set_{G,\cat F}^{\opname}} \cat C(X) \simeq \lim_{G/H \in \Orb_{G,\cat F}^{\opname}} \cat C(H).
    \]
\end{Lem}

\begin{proof}
By definition of coproduct completion, restriction along $i$ induces an equivalence
\[
\Fun^{\prod}(\Set_{G,\cat F}^{\opname}, \PrL)\xrightarrow{\sim}\Fun(\Orb_{G,\cat F}^{\opname}, \PrL)
\]
It then follows that the map on limits is an equivalence.
\end{proof}

 We now describe this limit as a totalization. 

 \begin{Lem}\label{lem2}
     Fix a set of representatives of $G$-conjugacy classes of subgroups of $\cat F$, and set 
     \[
     Y= \coprod_{(H) \in \cat F} G/H \in \Set_{G,\cat F}.
     \]
     The simplicial object  $Y^{\times \bullet +1}\colon \Delta^{\opname}\to\Set_{G,\cat F}$ is cofinal. In particular, there is an equivalence
     \[
     \lim_{X \in \Set_{G,\cat F}^{\opname}} \cat C(X)\simeq \mathrm{Tot}( \cat C(Y^{\bullet+1})).
     \]
 \end{Lem}

 \begin{proof}
     Note that any $X \in \Set_{G,\cat F}$ admits a map $X \to Y$. Then apply \cite[Proposition 6.28]{MNN2017}.
 \end{proof}

\begin{Cons}\label{cons-left-adjointable}
 We observe that there is a canonical functor
 \[
 \res\colon \cat C(G) \to \prod_{(H) \in \cat F} \cat C(H)\eqqcolon \cat C(Y) 
 \] 
 which is induced by the restriction functors. The functor $\res$:
 \begin{enumerate}
     \item is conservative; this follows from \cref{hyp-functor}(c).
     \item it admits a left and right adjoint: this follows from \cite[Theorem B and Corollary 1.3]{HY} and the fact that the restriction functors $\res^G_H$ are assumed to have both a left and right adjoint.
     \item It defines an augmentation for the cosimplicial object $\cat C(Y^{\bullet +1})$, meaning that there is a commutative diagram 
 \[
 \begin{tikzcd}
     \cat C(G) \arrow[r,"\res"] \arrow[d,"\res"'] & \cat C(Y) \arrow[d,"d^1"]\\
     \cat C(Y) \arrow[r,"d^0"] & \cat C(Y^{\times 2 })
 \end{tikzcd}
 \]
 which follows from the functoriality of $\cat C$.
 \end{enumerate} 
 \end{Cons}

\begin{Prop}\label{thm-stmod-lim}
     The functor induced by $\res$ 
     \[
     \cat C(G)\xrightarrow{\sim} \mathrm{Tot}(\cat C(Y^{\bullet+1}))
     \]
     is an equivalence. 
 \end{Prop}

 \begin{proof}
 We will verify the conditions of \cite[Corollary 4.7.5.3]{HA}. We have already seen that $\res$ is conservative and that it preserves all colimits. It is also clear that $\cat C(G)$ is presentable so (1) is satisfied. We only need to verify (2), that is for any $\alpha \colon [m] \to [n]$ in $\Delta_+$ the corresponding square
 \[
 \begin{tikzcd}
 \cat C(Y^{m+1}) \arrow[d,"\alpha"'] \arrow[r,"d^0"] & \cat C(Y^{m+2}) \arrow[d,"{[0]\star \alpha}"]\\
 \cat C(Y^{n+1}) \arrow[r,"d^0"] & \cat C(Y^{n+2})
 \end{tikzcd}
 \]
is left adjointable. This in turn will follow from \cref{rem-G-presentable} if we show that the corresponding square of $G$-spaces
\[
\begin{tikzcd}
 Y^{n+2} \arrow[d,"{([0]\star \alpha)^*}"']\arrow[r,"d_0"] & Y^{n+1} \arrow[d,"{\alpha^*}"] \\
Y^{m+2} \arrow[r,"d_0"] & Y^{m+1}.
\end{tikzcd}
\]
is a pullback. We can split this final claim, in two cases: if $\alpha$ is the unique map $[-1] \to [0]$ or if $\alpha$ is a map in $\Delta$. 

In the first case, we need to check that the square of $G$-spaces
\[
\begin{tikzcd}
Y^{\times 2} \arrow[r, "d_0"] \arrow[d,"d_1"'] & Y \arrow[d]\\
Y \arrow[r] & G/G
\end{tikzcd}
\]
is a pullback, which is clear. 

For the second case, we just need to check it for $\alpha$ a face $d^i$ and degeneracy $s^j$. In this case $[0]\star d^i=d^{i+1}$ and $[0]\star s^j=s^{j+1}$. The map induced by $d^i$ is projection onto all factors but avoiding $i+1$ and the map induced by $s^j$ is given by doubling the $j+1$ factor via a diagonal. Now is routine to check that the above square is a pullback.
 \end{proof}
 We have now all the ingredients to prove the main result of this section.

\begin{proof}[Proof of \cref{thm-descent}]
    For the first claim combine \cref{lem1}. \cref{lem2} and \cref{thm-stmod-lim}. For the second one, we note that if $\cat C$ lifts to a functor in $\CAlg(\PrL)$, then the restriction functors are automatically symmetric monoidal. 
\end{proof}

\subsection{Examples}\label{sec-descent-ex}
As an immediate consequence of \cref{thm-descent}, we recover the following result, which is by now folklore. Consider the functor from \cref{cor-funct-der}. 

\begin{Thm}\label{thm-descent-derivedcat}
Let $k$ be a commutative ring and let $G$ be a discrete group. Then there is a symmetric monoidal equivalence:
    \[
    \mathbf{D}(kG)\simeq \mathrm{Fun}(BG,\mathbf{D}(k))
    \] 
where the monoidal structure on the right-hand side is given pointwise.     
\end{Thm}

\begin{proof} 
 It is well-known that the restriction functor admits a left adjoint which satisfy the Mackey formula, see for example \cite[Lemma 3.2]{MPB06}. Note that also that the restriction functor to the trivial group $\res^G_1$ is clearly conservative, hence  the functor $\mathbf{D}(k-)$ satisfies \cref{hyp-functor} for the family $\cat F=\{1\}$. By \cref{thm-descent} we obtain a symmetric monoidal equivalence    
 \[
 \mathbf{D}(kG) \simeq \lim_{G/H \in \Orb_{G,\{1\}}^{\opname}} \mathbf{D}(k)\simeq \lim_{BG}\mathbf{D}(k)= \mathrm{Fun}(BG, \mathbf{D}(k))
 \]
which gives us the desired result.
\end{proof}

We need some preparations in order to provide further applications of the main theorem of this section.  

\begin{Prop}\label{prop-detection-abelem}
    Let $k$ be as in \cref{hyp} and $G$ be a finite group. Let $\mathcal{E}(G)$ denote the family of elementary abelian subgroups of $G$. Then both collections  
    \[
    \{\res^G_H\colon \mathbf{K}(\mathrm{Proj}(kG))\to \mathbf{K}(\mathrm{Proj}(kE))\mid E\in \mathcal{E}(G)\}
    \] 
    and  
     \[
     \{\res^G_H\colon \mathbf{StMod}(kG)\to \mathbf{StMod}(kE)\mid E\in \mathcal{E}(G)\}
     \]
     are jointly conservative. Moreover, if $k$ is a field of positive characteristic $p$, then we can replace $\cat E(G)$ by the collection of elementary abelian $p$-subgroups of $G$. 
\end{Prop}

\begin{proof}  
  Note that it is enough to verify the claim for this second collection of functors, since we may identify the stable category $\mathbf{StMod}(kG)$ with the subcategory $\mathbf{Ac}(\Proj(kG))$ of $\mathbf{K}(\Proj(kG))$, exactly as in \cref{prop-kproj-conservativity}. Recall that the homotopy category $\mathbf{StMod}(kG)$ coincides with the stable category of the Frobenius category $\mathrm{GP}(kG)$. In particular, a Gorenstein projective module is trivial in $\mathbf{StMod}(kG)$ if and only if it is projective. On the other hand, we can invoke Chouinard's theorem \cite{Chouinard} which tells us that projectivity of a $kG$-module is detected via restriction to all elementary abelian subgroups.

If $k$ is a field of positive characteristic $p$, we use the fact that the group algebra $kE$ is semisimple whenever $|E|$ is not divisible by $p$. Hence the only subgroups in $\cat E(G)$ that contribute non-trivially are the $p$-subgroups.
\end{proof}

\begin{Thm}\label{thm-descent-stmod}
Let $k$ be as in \cref{hyp}, let $G$ be a group of type $\mathrm{H}\Phi_k$, and let $\cat F$ be a collection of subgroups of $G$ containing the elementary abelian subgroups.  Then there is a symmetric monoidal equivalence:
    \[
    \mathbf{StMod}(kG)\simeq \lim_{G/H \in \Orb_{G,\cat F}^{\opname}} \mathbf{StMod}(kH).
    \]
\end{Thm} 

\begin{proof}
   The fact that the restriction functor admits a left adjoint is discussed in \cref{prop-induction-stmod}, and the Mackey formula follows again from the case of $kG$-modules (see for instance \cite[Lemma 3.2]{MPB06}) since the restriction-induction functors are Quillen functors on $kG$-modules. Conservativity of the restriction functors follows by combining \cref{prop-detection-abelem} and \cref{prop-conservativity for stab}.
\end{proof}

For the next result consider the functor from \cref{cor-functor-kproj}.
\begin{Thm}\label{thm-descent-kproj}
Let $k$ be as in \cref{hyp}, $G$ be a group of type $\Phi_k$, and let $\cat F$ be a collection of subgroups of $G$ containing the elementary abelian subgroups. Then there is an equivalence:
    \[
    \mathbf{K}(\Proj(kG))\simeq \lim_{G/H \in \Orb_{G,\cat F}^{\opname}} \mathbf{K}(\Proj(kH)).
    \]
\end{Thm}

\begin{proof}
   The fact that the restriction functor admits a left adjoint is discussed in \cref{prop-ind-res-kproj-mod}, and the Mackey formula follows again from the case of $kG$-modules. Conservativity of the restriction functors follows by combining \cref{prop-detection-abelem} and \cref{prop-kproj-conservativity}.
\end{proof}

\begin{Rem}
    Let $k$ be a field of positive characteristic $p$. In light of \cref{prop-detection-abelem}, we may instead restrict attention to families of subgroups $\cat F$ that contain all elementary abelian $p$-subgroups of $G$ in order to obtain descent results analogous to those in \cref{thm-descent-kproj} and \cref{thm-descent-stmod}. In the finite case, this recovers results on the stable module category from \cite{Mat16}, and for groups of type $\Phi_k$ it recovers those of \cite{Gomez24}. We emphasize that the corresponding results for $\mathbf{K}(\mathrm{Proj}(kG))$ appear to be new.
\end{Rem} 

As an application of the previous result, we obtain that  $\mathbf{K}(\mathrm{Proj}(kG))$ is indeed symmetric monoidal. We need some preparations. Recall the functor $\mathbf{t}$ from \cref{recollement for kG}.

\begin{Rec}\label{rec-complete-res}
A \textit{complete resolution} of a $kG$--module $M$ is a morphism of complexes $\mathbf{t}M \to \mathbf{p}M$ such that $\mathbf{t}M$ is a totally acyclic complex of projective $kG$--modules, $\mathbf{p}M$ is a projective resolution of $M$, and the map $\mathbf{t}M \to \mathbf{p}M$ agrees with the identity in sufficiently low degrees. Moreover, complete resolutions are unique up to chain homotopy, and any morphism of projective resolutions induces, up to homotopy, a unique morphism between the corresponding complete resolutions; see \cite[Definition 3.8]{mazza2019stable} and the subsequent discussion.

In particular, if $k$ satisfies \cref{hyp} and $G$ is a group of type $\Phi_k$, then every $kG$--module admits a complete resolution; see \cite[Theorem 3.9]{mazza2019stable}.
\end{Rec}

\begin{Rec}\label{rec-monoidalunit-finite}
    Let $G$ be a finite group and $k$ be as in \cref{hyp}. In this case, the complete resolution $\mathbf{t}k\to \mathbf{p}k$ of the trivial module $k$ corresponds to the Tate resolution $\mathbf{t}k$ of $k$, a projective resolution of $\mathbf{p}k$ of $k$ and the chain map $\mathbf{t}k\to \mathbf{p}k$ is simply the identity map on non-positive degrees. From this, we deduce that the cofiber $\mathbf{i}k$ of $\mathbf{t}k\to \mathbf{p}k$ is precisely an injective resolution of $\mathbf{k}$ in the category of lattices $\mathrm{Mod}(G,k)$ (see \cite[Section 2]{BBIKP}). Moreover, $\mathbf{i}k$ is the monoidal unit of $\mathbf{K}(\mathrm{Proj}(kG))$; see \cite[Lemma 4.2]{BBIKP}.
\end{Rec}

\begin{Cor}\label{cor-kproj-ttcat}
   Let $k$ be as in \cref{hyp} and $G$ be a group of type $\Phi_k$. Then the tensor product $\otimes_k$ induces on  $\mathbf{K}(\Proj(kG))$ a symmetric monoidal structure where the monoidal unit is given by $\mathbf{i}k$, the cofiber of $\mathbf{t}k\to \mathbf{p}k$, the total resolution of the trivial $kG$--module $k$. 
\end{Cor}

\begin{proof} 
Recall from \cref{def-kproj} that $\otimes_k$ induces a non-unital symmetric monoidal structure on $\mathbf{K}(\Proj(kG))$. Since the restriction functors preserve this tensor product (by \cref{prop-res-motimesn}), by the universal property of the limit, we obtain a non-unital symmetric monoidal functor 
\[
 \mathbf{K}(\Proj(kG))\xrightarrow{\sim}\lim_{G/H \in \Orb_{G,\cat F}^{\opname}} \mathbf{K}(\Proj(kH)).
\]
which is an equivalence by \cref{thm-descent-kproj}. Since limits of symmetric monoidal categories is again symmetric monoidal, we deduce that the non-unital symmetric monoidal structure $\otimes_k$ on $\mathbf{K}(\Proj(kG))$ refines to a symmetric monoidal structure. We now verify that the unit object is as claimed. 

Let $H$ be a finite subgroup of $G$ and let $\mathbf{t}_Hk\to \mathbf{p}_Hk$ be a complete resolution of the trivial $kH$--module $k$.  Since the restriction functor is exact and preserves projective $kG$--modules, we deduce that $\res^G_H(\mathbf{t}k\to \mathbf{p}k)$ is a complete resolution of $\res^G_Hk$. On the other hand, note that $\res^G_H k$  is simply the trivial $kH$--module. Hence the identity map $\res^G_Hk$ to $k$ induces a homotopy equivalence between $\res^G_H\mathbf{p}k$ and $\mathbf{p}_H k$, and so it does on complete resolutions (see \cref{rec-complete-res}). It follows that the cofiber of $\res^G_H(\mathbf{t}k\to \mathbf{p}k)$ is canonically homotopy equivalent to the cofiber of $\mathbf{t}_Hk\to \mathbf{p}_Hk$ which is the monoidal unit of $\mathbf{K}(\mathrm{Proj}(kH))$ (see \cref{rec-monoidalunit-finite}).  

In other words, we have verified that for any finite subgroup $H$ of $G$, $\res^G_H\mathbf{i}k$ is canonically isomorphic to the monoidal unit of $\mathbf{K}(\mathrm{Proj}(kH))$ via the induced chain map induced by the identity map from  $\res^G_H k$ to $k$. It follows that $(\res^G_H \mathbf{i}k)_{H}$ agrees with the monoidal unit of $\lim_{G/H \in \Orb_{G,\fin}^{\opname}} \mathbf{K}(\Proj(kH))$. It follows that $\mathbf{i}k$ must be the monoidal unit of $\mathbf{K}(\mathrm{Proj}(kG))$ by \cref{thm-descent-kproj}.
\end{proof}

Recall the functor from \cref{cor-funct-dperm}.
\begin{Thm}\label{thm-descent-permutation}
Let $k$ be a commutative Noetherian ring and let $G$ be a group admitting a finite-dimensional model for $\underline{E}G$. Then there is a symmetric monoidal equivalence:
    \[
    \DPerm^{\fin}(G, k)\simeq \lim_{G/H \in \Orb_{G,\fin}^{\opname}} \DPerm(H,k).
    \]
\end{Thm}

\begin{proof}
   The fact that the restriction functor admits a left adjoint is discussed in \cref{prop-ind-res-permutation-mod}, and the Mackey formula follows again from the case of $kG$-modules, see \cite[Lemma 3.2]{MPB06}. Conservativity of the restriction functors follows from \cref{rem:conservativity}.
\end{proof}

\begin{Rem}
    We warn the reader that, as opposed to its counterparts $\mathbf{StMod}(kG)$ and $\mathbf{K}(\mathrm{Proj}(kG))$,  the derived category of finite permutation modules does not satisfy descent to elementary abelian subgroups. Indeed, this was already observed in \cite[Remark 6.4]{Gomez2025} for the group $G=C_4$. In this case, the restriction functor 
    \[
    \res^{C_4}_{C_2}\colon\DPerm^{\fin}(C_4, k)\to  \DPerm(C_2,k) 
    \]
    is not conservative. 
\end{Rem}

On the positive side, one may instead descend the derived category of finite permutation modules to the family of all subgroups of $p$-power order, allowing the prime $p$ to vary:

\begin{Cor}
  Let $k$ be a commutative Noetherian ring and let $G$ be a group admitting a finite-dimensional model for $\underline{E}G$. Let $\cat {F}$ denote the family of $p$-power order finite subgroups of $G$, allowing the prime $p$ to vary.
  Then there is a symmetric monoidal equivalence:
    \[
    \DPerm^{\fin}(G, k)\simeq \lim_{G/H \in \Orb_{G,\cat F}^{\opname}} \DPerm(H,k).
    \]
\end{Cor}

\begin{proof}
   We only need to verify that the family $\{\res^G_H\mid H \in \cat F\}$ is jointly conservative. Since $\{\res^G_H\mid H \subseteq G \; \mathrm{finite}\}$ is jointly conservative by \cref{rem-dperm is comp}, and $\res^G_H \circ \res^H_L \simeq \res^G_L$ for any chain of subgroups $L \subseteq H \subseteq G$, it follows that it suffices to prove the claim in the case where $G$ is finite. This is precisely \cite[Proposition 6.3]{Gomez2025}.
\end{proof}

\begin{Rem}
In particular, the previous result  may reduce, to some extend, the study of the category $\DPerm^{\fin}(G,k)$ to the case where $G$ is a finite $p$-group.    
\end{Rem}

\subsection{Groups acting on trees}

We now focus on groups acting on trees with finite isotropy. By Bass–Serre theory, such groups correspond to fundamental groups of graphs of finite groups \cite{DD89}, and we will use this perspective throughout this section. Let us begin with an example.

\begin{Exa}\label{Ex-Orb-SL2}
    Let $G=\mathrm{SL}_2(\mathbb Z)=C_6\ast_{C_2}C_4$. In this case, the orbit category $\Orb_{G,\fin}$ for the family of finite subgroups has a cofinal subcategory  whose objects are  $G/C_2$, $G/C_4$ and $G/C_6$, together  with two $G$-equivariant maps $G/C_2\to G/C_4$ and $G/C_2\to G/C_6$; see, for instance, \cite[Examples 2.6]{barcenas2023survey}. In particular, for $\mathrm{SL}_2(\Z)$ all the examples from the previous section admit a more appealing description as pullbacks.   
\end{Exa}

In fact, for the stable module category, one can generalize the previous example in a suitable sense. 

\begin{Exa}\label{Ex-stmod-trees}
   Let $G$ be a group acting on a tree with finite isotropy, and let $k$ be as in \cref{hyp}. Then the stable module category $\mathbf{StMod}(kG)$ satisfies descent with respect to the graph of groups associated to $G$. More explicitly, let $\Gamma_G \colon \Gamma \to \mathbf{Grp}$ denote the graph of groups associated to $G$. Then the restriction functors induce an equivalence of stable $\infty$-categories
    \[
    \mathbf{StMod}(kG)\xrightarrow[]{\simeq} \lim_{\sigma\in\Gamma} \mathbf{StMod}(kG_\sigma)
    \]
    where $G_\sigma$ denotes the value of $\Gamma_G$ at $\sigma\in \Gamma$. We refer to \cite[Section 3]{Gomez24} for further details. 
\end{Exa}

Note that in \cref{Ex-Orb-SL2}, this full subcategory is precisely the one associated with the graph of groups determining $\mathrm{SL}_2(\mathbb Z)$. Motivated by these examples, one may ask whether the other functors considered in the previous section also admit a descent description in terms of the graph of groups. We leave this as an open question.

\begin{question}
  Let $G$ be a group acting on a tree with finite isotropy. Under which additional assumptions on $G$ do the categories $\mathbf{K}(\Proj(kG))$ and $\mathbf{DPerm}^{\fin}(G,k)$ admit a description in terms of the associated graph of groups?
\end{question}

\begin{Rem}
    Of course, one should not expect this to hold in full generality. For instance, let $G = \mathbb{Z}$. In this case, $G$ is an HNN extension, and its associated graph of groups consists of a single vertex with a loop, with the trivial group assigned everywhere. In particular,
    \[
    \lim_{\sigma\in\Gamma} \mathbf{K}(\Proj(kG_\sigma)) \simeq \mathbf{D}(k)
    \]
    which does not agree with $\mathbf{K}(\Proj(kG))$.  
    \end{Rem}

\section{Finite permutation resolutions}\label{sec-perm-resolutions}

For finite groups and over a regular ring (which is assumed to be Noetherian, see \cref{rem-regular-rings}), we know by a result of Balmer-Gallauer that the homotopy category of projectives is a finite localization of the derived category of permutation modules. In this section, we show that a similar result also holds for groups with a finite-dimensional model for $\underline{E}G$.

\begin{Cons}
    We note that for a finite group $H$, we have a natural functor
    \[
    q_H \colon \DPerm(H,k)^c \to \mathbf{D}^b(\mod(kH))
    \]
    which is induced by the inclusion $\mathrm{perm}(H,k)\subseteq \mod(kH)$. In fact the above functors assemble into a natural transformation 
    \[
    (q \colon  \DPerm(-,k)^c \to \mathbf{D}^b(\mod(k-)) \in \Fun(\Orb_{G,\fin}^{\opname},\CAlg(\Cat^{\perf}_\infty)).
    \]
    Passing to the Ind-completion we obtain a natural transformation between big categories. If $k$ is regular, we know that  
    \[
    \Ind(\mathbf{D}^b(\mod(kH)))=\mathbf{K}(\Proj(kH))
    \]
    by \cite[Propositions 2.14 and 2.16]{BBIKP}. Therefore we have constructed a natural transformation 
    \[
    (q \colon  \DPerm(-,k) \to  \mathbf{K}(\Proj(k-))\in \Fun(\Orb_{G,\fin}^{\opname},\CAlg(\PrL)).
    \]
\end{Cons}
The starting point is the following result. 
\begin{Thm}[Balmer-Gallauer]\label{prop-radjoint-fincase}
    Let $H$ be a finite group and $k$ a regular  commutative ring. Then the canonical functor 
    \[
    q_H \colon \DPerm(H,k) \to \mathbf{K}(\Proj(kH))
    \]
    is a finite localization. 
\end{Thm}
\begin{proof}
    Our functor $q$ is precisely the ind-completion of the functor constructed in \cite[Theorem 1.4]{balmer2023finite}, and so it is a finite localization. 
\end{proof}

\begin{Cons}\label{cons-q}
    Let $G$ be a group admitting a finite-dimensional model for $\underline{E}G$ and let $k$ be as in \cref{hyp}. In this setting we have a functor 
    \[
    q_G \colon \DPerm^{\fin}(G,k) \to \mathbf{K}(\Proj(kG)) 
    \]
    which is defined via the following commutative square 
    \begin{equation}\label{def-qG}
    \begin{tikzcd}
         \DPerm^{\fin}(G,k)\arrow[r, "q_G"] \arrow[d, "\res"',"\sim"] & \mathbf{K}(\Proj(kG)) \arrow[d,"\res", "\sim"']\\
        \lim_H\DPerm^{\fin}(H,k) \arrow[r,"\lim_H q_H"] & \lim_H \mathbf{K}(\Proj(kH)) 
    \end{tikzcd}
    \end{equation}
    where we used our descent results \cref{thm-descent-kproj} and \cref{thm-descent-permutation}. 
\end{Cons}

With this result in hands, we can deduce a similar statement for more general groups. 

\begin{Thm}\label{thm-perm-localization-kproj}
    Let $G$ be a group admitting a finite-dimensional model for $\underline{E}G$ and $k$ as in \cref{hyp}. Then the canonical functor 
    \[
    q_G \colon \DPerm^{\fin}(G,k) \to \mathbf{K}(\Proj(kG))
    \]
    is a finite localization. 
\end{Thm}
\begin{proof}
   Consider the commutative square \eqref{def-qG}. By taking horizontal fibers and noting that fibers commutes with limits, we obtain an equivalence 
   \[
   \Ker(q_G)\simeq \lim_H \Ker(q_H).
   \]
   \cref{prop-radjoint-fincase} implies that $\Ker(q_H)$ is a smashing ideal for any finite subgroup $H\subseteq G$. Since taking smashing ideals commutes with limits by \cite[Corollary 9.9]{descent}, we deduce that $\Ker(q_G)$ is a smashing ideal too, and so $q_G$ is a smashing localization. We now verify that it is a finite localization. 
   Since the inclusion $\Ker(q_G) \subseteq \DPerm^{\fin}(G,k)$ preserves compact objects (because $q_G$ is a smashing localization), it will be enough to verify that the kernel is compactly generated. To this end, we apply \cref{prop-compactgen} to the projection functors 
   \[
   \{p_H \colon \lim_H \Ker(q_H) \to \Ker(q_H)\}_{H}.
   \]
   Condition (c) of \cref{prop-compactgen} is satisfied by construction, and (b) holds by \cref{prop-radjoint-fincase}. Thus, the only condition that requires discussion is (a). For this, we apply \cite[Proposition 3.14]{Sil} to the functor 
   \[
   \Orb_{G,\fin}^{\opname}\to \CAlg(\PrL), \quad G/H \mapsto \Ker(q_H)
   \]
   with marking $\mathcal{W}=\Orb_{G,\fin}^{\opname}$ so that the partially lax limit agrees with the limit. Note that the proposition applies since the transition functors (which in our case are given by the restriction functors) preserves all limits. It then follows from the cited result that $\lim_H \Ker(q_H)$ admits all limits, and the discussion in \cite[Remark 3.15]{Sil} shows that these limits are computed pointwise. All together this implies that the projection functors $p_H$'s preserve limits and so by the adjoint functor theorem they admit a left adjoint, giving (a).  
\end{proof}
 As a consequence of our work, we obtain the following generalization of \cite{balmer2023finite}:

\begin{Cor}\label{coro-finite-perm-res}
     Let $G$ be a group admitting a finite-dimensional model for $\underline{E}G$ and $k$ as in \cref{hyp}. Then any  $kG$--module $M$ of type $FP_\infty$ is a retract of a $kG$--module that has a finite resolution by finite $\natural$-permutation $kG$-modules; that is,  modules in $\mathrm{perm}^{\fin}(G,k)^\natural$. 
\end{Cor}

\begin{proof}
The finite localization functor $q_G$ from  \cref{thm-perm-localization-kproj} induces a functor on compact objects, which, thanks to \cref{cor-compact-Kproj} and \cref{rem-dperm is comp}, we can rewrite as 
\[
q_G^c\colon \mathbf{K}^b(\mathrm{perm}^{\fin}(G,k)^\natural)  \to \mathbf{D}^{b}(\mathrm{mod}_\infty(kG)).
\]
By Neeman–Thomason localization theorem \cite[Theorem 2.1]{Neeman92}, this map is a Verdier quotient up to idempotent completion:
\[
\left(\frac{\mathbf{K}^b(\mathrm{perm}^{\fin}(G,k)^\natural)}{\text{acyclic cxs}} \right)^\natural\xrightarrow{\sim} \mathbf{D}^{b}(\mathrm{mod}_\infty(kG)).
\]
Combining this with the canonical inclusion
\[
\mathrm{mod}_\infty(kG) \subseteq \mathbf{D}^{b}(\mathrm{mod}_\infty(kG))
\]
gives the claim. 
\end{proof}

We conclude this section with the following open question and a related remark.

\begin{question}\label{ques-1}
Let $G$ be a group admitting a finite-dimensional model for $\underline{E}G$, and let $k$ be a field of positive characteristic. Does every $kG$--module of type $FP_\infty$ admit a finite resolution by finite $\natural$-permutation $kG$--modules?
\end{question}

\begin{Rem}
For finite groups over a field of positive characteristic, Balmer--Gallauer proved in \cite{balmer2023finite} that the idempotent completion of the Verdier quotient associated to the functor $q^c_G$ is unnecessary. Consequently, one obtains a positive answer to \cref{ques-1} in this case.

It is unclear whether the same statement holds for infinite groups, and it is not evident how to adapt the methods of \textit{loc. cit.} to our setting (one of the problem being that $K_0(\mod_\infty(kG))$ is not a ring in general). In particular, we are not aware of analogous $K$-theoretic properties being established for infinite groups.
\end{Rem}

\section{Picard groups of permutation modules}\label{sec-picard}

In this section we give some computations of the Picard group of permutation categories associated to groups with a finite-dimensional model for the classifying space for proper actions with special focus on groups acting on trees with finite isotropy. 

Let $\cat C$ an arbitrary symmetric monoidal $\infty$-category. We briefly recall some facts about the Picard spectrum associated to $\cat C$. For further details, we refer to \cite{MS16}. 

\begin{Rec}
    Recall that the \textit{Picard group $\mathbf{Pic}(\cat C)$} of $\cat C$ is the the group of iso-classes of invertible objects with multiplication given by the tensor product. An object $x$ in $\cat C$ is \textit{invertible} if there is $y$ in $\cat C$ such that $x\otimes y\simeq \mathbb 1$. In fact, one can associate a connective  spectrum $\mathrm{pic}(\cat C)$ which has $\mathbf{Pic}(\cat C)$ as one of its homotopy groups. Explicitly, there is a functor 
     \[
     \mathrm{pic}(-)\colon \mathrm{Cat}^\otimes\to \Sp_{\geq0} 
     \]       
given as the composition $\Omega ^\infty \mathrm{Pic}$ where $\mathrm{Pic}(\cat C)$ is given as the space of invertible objects in $\cat C$ and equivalences between them which is a group-like space. In particular, we have that 
\[
\pi_i\mathrm{pic}(\cat C) = \left\{
        \begin{array}{ll}
           
            \mathbf{Pic}(\cat C) & \textrm{ if } n=0\\
            (\pi_0 \mathrm{Hom}_{\cat C}(\mathbb{1},\mathbb{1}))^\times & \textrm{ if } n=1 \\
            \pi_{i-1} \mathrm{Hom}_{\cat C} (\mathbb{1},\mathbb{1}) & \textrm{ if } n\geq 2.
        \end{array}
    \right.
\]
\end{Rec}

\begin{Rem}\label{rem-spectral sequence for lim}
    Our interest in the functor  $\mathrm{pic}$ lies in the fact that it commutes with limits and filtered colimits; see \cite[Proposition 2.2.3]{MS16}. In particular, given a functor  $\cat C_\bullet\colon  I^{\opname} \to\mathrm{Cat}^\otimes$, there is a Bousfield-Kan spectral sequence 
\begin{equation}\label{spectral sequence for Pic}
E^{p,q}_2=H^p(I;\pi_q \mathrm{pic}(\cat C_\bullet ))\Rightarrow \pi_{q-p}\left(\varprojlim \mathrm{pic}\left(\lim_I \cat C_\bullet\right)\right).    
\end{equation}  
 In particular, this spectral sequence computes the Picard group of $\lim_I \cat C_\bullet$. Note that this spectral sequence has no convergence issues as the classical 
 Bousfield-Kan spectral sequence for homotopy limits since $\mathrm{pic}$ is connective; see \cite[Section 1.2.2]{HA}. 
\end{Rem}

\begin{Cor}\label{thm-ses for Dperm}
    Let $G$ be a group with finite-dimensional model for $\underline{E}G$ and $k$ be commutative Noetherian ring. Consider the functor  
    \[
    \DPerm^{\fin}_\bullet\colon \Orb_{G,\fin}^{\opname} \to \CAlg(\PrL)
    \]
    introduced in \cref{cor-funct-dperm}.  Then there is an exact sequence of abelian groups
    \begin{center}
    \adjustbox{scale=0.80,center}{%
    \begin{tikzcd}
0\arrow[r] & H^1(\Orb_{G,\fin};\pi_1 \mathrm{pic}\DPerm^{\fin}_\bullet) \arrow[r]
&   \mathbf{Pic} (\DPerm^{\fin}(G,k)) \arrow[d, phantom, ""{coordinate, name=Z}]
\arrow[d,
rounded corners,"",
to path={ -- ([xshift=4ex]\tikztostart.east)
 |- (Z) [near start]\tikztonodes
-|([xshift=-4ex]\tikztotarget.west)
-- (\tikztotarget)}] &
 \\
& &  H^0(\Orb_{G,\fin};\pi_0 \mathrm{pic}\DPerm^{\fin}_\bullet) \arrow[r] & H^2(\Orb_{G,\fin};k^\times) .
\end{tikzcd}}
\end{center}  
\end{Cor}

\begin{proof}
    By \cref{thm-descent-permutation} we have that 
    \[
    \DPerm^{\fin}(G,k)\simeq \lim_{\Orb_{G,\fin}^{\opname}}\DPerm^{\fin}_\bullet
    \]
    and hence we obtain a spectral sequence which computes the Picard group of $\DPerm^{\fin}(G,k)$; see \cref{rem-spectral sequence for lim}. Then result now follow by noticing that the presheaf 
    \[
    \pi_j \mathrm{pic}\DPerm^{\fin}_\bullet\colon \Orb_{G,\fin}^{\opname} \to \mathrm{Ab}
    \]
    is trivial for all $j>1$. Indeed, notice that there are no non-trivial higher extensions of the monoidal unit of $\DPerm^{\fin}(H,k)$ for any finite subgroup $H\subseteq G$.  Hence the only possible non trivial differential that is relevant for the Picard group is $d^{0,0}_2\colon E^{0,0}\to E^{2,1}$ which gives  us the result. 
\end{proof}

\begin{Rem}
    We stress that the Picard group of $\DPerm^{\fin}(G,k)$ for $G$ finite and $k$ a field is known; see \cite{Mil25}. Hence the previous result can be used to give  explicit computations in this context as we will see. 
\end{Rem}

\begin{Exa}
    Let $G=\mathrm{SL}_2(\mathbb Z)$, and $k$ be a field. Recall from \cref{Ex-Orb-SL2}, that the category $\mathbf{Orb}_{\mathrm{SL}_2(\mathbb Z),\fin}$ has three elements, namely $\mathrm{SL}_2(\mathbb Z)/C_6$, $\mathrm{SL}_2(\mathbb Z)/C_4$ and $\mathrm{SL}_2(\mathbb Z)/C_2$, and two equivariant maps $\mathrm{SL}_2(\mathbb Z)/C_2\to \mathrm{SL}_2(\mathbb Z)/C_6$ and $\mathrm{SL}_2(\mathbb Z)/C_2\to \mathrm{SL}_2(\mathbb Z)/C_4$. Moreover, the presheaf \[
    \pi_1 \mathrm{pic}\DPerm^{\fin}_\bullet\colon \Orb_{\mathrm{SL}_2(\mathbb Z),\fin}^{\opname} \to \mathrm{Ab}
    \] 
    is constant with value $k^\times$. It follows that  
    \[
    H^i(\Orb_{\mathrm{SL}_2(\mathbb Z),\fin};\pi_1 \mathrm{pic}\DPerm^{\fin}_\bullet)=H^i(|\Orb_{\mathrm{SL}_2(\mathbb Z),\fin}|; k^\times)=0
    \]
    for all $i>0$,  using that the nerve of the orbit category has the homotopy type of a segment. 
    Hence, we obtain an equivalence of abelian groups 
    \[
    \mathbf{Pic}(\DPerm^{\fin}(\mathrm{SL}_2(\mathbb Z),k))\simeq  H^0(\Orb_{\mathrm{SL}_2(\mathbb Z),\fin};\pi_0 \mathrm{pic}\DPerm^{\fin}_\bullet)
    \]
    where the right hand side is just the limit of the presheaf diagram. 
    We can now explicitly compute this limit using  \cite[Theorem 6.5]{miller2024endotrivial} and \cite[Theorem 6.2]{Mil25}: 
    \begin{enumerate}
        \item If $k$ has characteristic coprime to $2$ and $3$, then 
        \[
        \mathbf{Pic}(\DPerm^{\fin}(H,k))\simeq \langle k[1]\rangle\simeq \mathbb{Z}
        \]
        for any finite subgroup $H\subseteq \mathrm{SL}_2(\mathbb Z)$. Hence $\mathbf{Pic}(\DPerm^{\fin}(H,k))\simeq \mathbb Z$. 
        \item If $k$ has characteristic 2, then 
        \[
\mathbf{Pic}( \DPerm^{\fin}(H,k)) \simeq \left\{
        \begin{array}{ll}
           
            \mathbb{Z}^2=\langle k[1], u_2 \rangle & \textrm{ if } H=C_2 \\
            \mathbb{Z}^2=\langle k[1], \mathrm{Infl}_{C_6/C_3}^{C_6} u_2 \rangle & \textrm{ if } H=C_6 \\
             \mathbb{Z}^2=\langle k[1], \mathrm{Infl}_{C_4/C_2}^{C_4} u_2, \rangle & \textrm{ if } H=C_4.
        \end{array}
    \right.
\]
where $u_2$ is the complex of $kC_2$--modules given as  the brutal truncation in degree 1 of the minimal resolution of $k$ by free $kC_2$--modules. Moreover, note that 
\[
\res^{C_6}_{C_2}\mathrm{Infl}_{C_6/C_3}^{C_6} u_2=u_2 \, \mbox{ and } \, \res^{C_4}_{C_2}\mathrm{Infl}_{C_4/C_2}^{C_4} u_2=k[1].
\]
We deduce that $\mathbf{Pic}( \DPerm^{\fin}(\mathrm{SL}_2(\mathbb{Z}),k))\simeq \mathbb{Z}^2$. 
\item If $k$ has characteristic 3, we have that 
   \[
\mathbf{Pic}( \DPerm^{\fin}(H,k)) \simeq \left\{
        \begin{array}{ll}
           
            \mathbb{Z}=\langle k[1]\rangle & \textrm{ if } H=C_2,C_4 \\
             \mathbb{Z}^2\oplus \mathrm{Hom}(C_6,k^\times),  & \textrm{ if } H=C_6.
        \end{array}
    \right.
\]
where the generators of the torsion-free part in the last case correspond to $\langle k[1], \mathrm{Infl}_{C_3}^{C_6} u_3\rangle$ with $u_3$ the complex of $kC_3$--modules given as  the brutal truncation in degree 2 of the standard resolution of $k$ by free $kC_3$--modules and  $\Hom(C_6,k^\times)$ identifies with the one-dimensional representations. 
Note that 
\[
\res^{C_6}_{C_2}\mathrm{Infl}_{C_6/C_2}^{C_6} u_3=k[2].
\]
We deduce that $\mathbf{Pic}( \DPerm^{\fin}(\mathrm{SL}_2(\mathbb{Z}),k))\simeq \mathbb{Z}^2\oplus \mathrm{Hom}(C_6,k^\times)$.
    \end{enumerate}
\end{Exa}

A key input in the previous computation is the fact that the orbit category of  $\mathrm{SL}_2(\mathbb Z)$ for the family of finite subgroups is simply connected. Let us record this observation. 

\begin{Cor}\label{cor-pic for trees}
    Let $G$ be a group with a finite-dimensional model for $\underline{E}G$, and $k$ be a commutative ring. If $\mathrm{Orb}_{G,\fin}$ is acyclic, 
    then there is an isomorphism of abelian groups
    \[
    \mathbf{Pic}(\DPerm^{\fin}(G,k))\simeq \lim_{G/H \in \Orb_{G,\fin}^{\opname}}\mathbf{Pic}(\DPerm^{\fin}(H;k)).
    \]
\end{Cor}

\begin{proof}
    In view of \cref{thm-ses for Dperm} and the definition of $H^0$, we only need to verify that 
     \[
    H^i(\Orb_{G,\fin},\pi_1 \mathrm{pic}\DPerm^{\fin}_\bullet)=0 \; \mbox{ for } i=1,2. 
    \]
    But this follows since the presheaf  \[
    \pi_1 \mathrm{pic}\DPerm^{\fin}_\bullet\colon \Orb_{G,\fin}^{\opname} \to \mathrm{Ab}
    \] 
    is constant with value $k^\times$; and hence  
    \[
     H^i(\Orb_{G,\fin},\pi_1 \mathrm{pic}\DPerm^{\fin}_\bullet)\cong  H^i(|\Orb_{G,\fin}|,k^\times)
    \]
    which is trivial by our assumption.  Then the claim follows. 
\end{proof}

\section{Galois groups and separable algebras}\label{sec-galois}

In this section, we showcase how to use our descent result for the stable module category to compute its Galois group and to classify separable commutative algebras. This is not intended to be a comprehensive discussion, but rather to provide an indication of the types of calculations that can be performed using our methods. 
We start by recalling some background from \cite{Mat16} and \cite{NaumannPol}.
\begin{Rec}\label{rec-galois}
    Let $\cat C$ be a presentably symmetric monoidal stable $\infty$-category. Associated to $\cat C$ there are categories 
    \[
    \CAlg^{\mathrm{cov}}(\cat C) \subseteq \CAlg^{\mathrm{wcov}}(\cat C) \subseteq \CAlg^{\mathrm{sep},\mathrm{f}}(\cat C^{\dual})
    \]
    of finite covers, weak finite covers, and separable algebras of finite degree with underlying dualizable object. 
    If $\pi_0(\unit_{\cat C})$ is indecomposable, then the categories of (weak) finite covers are governed by profinite groups $\pi_{\leq 1}(\cat C)$ and $\pi_{\leq 1}^{\mathrm{weak}}(\cat C)$ in the sense that there are equivalences
    \[
    \CAlg^\mathrm{cov}(\cat C)^{\opname}= \Fin_{\pi_1(\cat C)}^{\mathrm{cts}} \quad \mathrm{and}\quad  \CAlg^\mathrm{wcov}(\cat C)^{\opname}= \Fin_{\pi_1^{\mathrm{wcov}}(\cat C)}^{\mathrm{cts}}.
    \]
    These profinite groups are refer to as the Galois group and the weak Galois groups of $\cat C$.
    We also recall that if $\unit \in \cat C$ is compact, then there is no difference between weak covers and finite covers. We can use our descent result to access these categories because the functors 
    \[
    \CAlg^{\mathrm{wcov}}(-) \quad \mathrm{and}\quad \CAlg^{\mathrm{sep}}((-)^{\dual})
    \]
    commutes with all limits. 
\end{Rec}
\begin{Exa}
    Let $k$ be a field of char $p$ and let $G$ a group of type $\Phi_k$. In $\mathbf{StMod}(kG)$ we can identify $\pi_0(\unit)$ with the Ikenaga-Farrel-Tate cohomology $\widehat{H}^0(G;k)$. This is in general not indecomposable, but it is in many cases of interest. For instance it is for groups such as $\mathrm{SL}_2(\Z)$ and $\Z/p^\infty$ by \cite[Theorem 5.5]{mazza2019stable}. 
\end{Exa}

\begin{Exa}
    Consider the group $G=\mathrm{SL}_2(\Z)=C_6 \ast_{C_2} C_4$. By \cref{Ex-stmod-trees} and \cref{Ex-Orb-SL2}, there is a pullback 
    \[
    \begin{tikzcd}
         \mathbf{StMod}(k \mathrm{SL}_2(\mathbb Z)) \arrow[r] \arrow[d] & \mathbf{StMod}(k C_6) \arrow[d]\\
         \mathbf{StMod}(kC_4) \arrow[r] & \mathbf{StMod}(k C_2).
    \end{tikzcd}
    \]
    Over a separably closed field $k$ of char $p$, we can apply \cite[Proposition 14.11]{NaumannPol} and deduce
    \begin{align*}
            \CAlg^{\mathrm{cov}}(\mathbf{StMod}(k C_6))&=\CAlg^{\mathrm{sep},\mathrm{f}}(\mathbf{StMod}(k C_6)^{\dual})\\
            \CAlg^{\mathrm{cov}}(\mathbf{StMod}(k C_4))&=\CAlg^{\mathrm{sep},\mathrm{f}}(\mathbf{StMod}(k C_4)^{\dual})\\
            \CAlg^{\mathrm{cov}}(\mathbf{StMod}(k C_2))&=\CAlg^{\mathrm{sep},\mathrm{f}}(\mathbf{StMod}(k C_2)^{\dual}).\\
        \end{align*}
    Note that there is no difference between finite covers and separable commutative algebras. Since this condition is stable under pullback, we get that 
    \[
    \CAlg^{\mathrm{cov}}( \mathbf{StMod}(k G))=\CAlg^{\mathrm{sep},\mathrm{f}}( \mathbf{StMod}(kG)^{\dual}).
    \]
    Therefore, we just need to classify the finite covers:
    \begin{enumerate}
        \item If char $k=3$, then $\mathbf{StMod}(kC_4)=\mathbf{StMod}(kC_2)=0$ from which we conclude that 
        \[
        \CAlg^{\mathrm{cov}}( \mathbf{StMod}(kG)))=\CAlg^{\mathrm{cov}}(\mathbf{StMod}(k C_6))=\Fin_{C_2}
        \]
        using \cite[Theorem 14.14]{NaumannPol}. In other words, the Galois group is $C_2$. 
        \item If char $k=2$, then using \cite[Theorem 14.14]{NaumannPol} again we find that 
        \begin{align*}
            \CAlg^{\mathrm{cov}}(\mathbf{StMod}(k C_6))& =\Fin_{C_3}\\
            \CAlg^{\mathrm{cov}}(\mathbf{StMod}(k C_4))& =\Fin_{C_2}\\
            \CAlg^{\mathrm{cov}}(\mathbf{StMod}(k C_2))&=\Fin\\
        \end{align*}
        Since taking finite covers preserves finite limits, we find that 
        \[
        \CAlg^{\mathrm{cov}}( \mathbf{StMod}(k G))=\Fin_{C_3} \times_{\Fin} \Fin_{C_2}.
        \]
        By the Galois correspondence \cite[Theorem 5.36 ]{Mat16}, we can translate this pullback to a pushout of profinite groups $C_3 \cup_1 C_2$: this is the profinite completion of $\mathrm{PSL}_2(\Z)=C_2 \ast C_3$. Applying again the Galois correspondence, we find that 
         \[
         \CAlg^{\mathrm{cov}}( \mathbf{StMod}(k G))^{\opname}=\Fin^{\mathrm{cts}}_{\widehat{\mathrm{PSL}_2(\Z)}}.
         \]
    \end{enumerate}
\end{Exa}

\begin{Exa}
    Let $k$ be a separably closed field of characteristic $p$ and consider the group $\Z/p^\infty$. Then by \cref{Ex-stmod-trees} we can write
    \[
    \mathbf{StMod}(k\Z/p^\infty )=\lim_n \mathbf{StMod}(k\Z/p^n ).
    \]
    For the terms in the limits there is no difference between weak finite covers and separable algebras of finite degree with dualizable underlying object (see \cite[Proposition 14.11]{NaumannPol}), so the same will hold for the limit (see \cite[Theorem 9.7]{NaumannPol}). Moreover we know by \cite[Theorem 9.18]{Mat16} that the weak Galois group of $\mathbf{StMod}(k\Z/p^n )$ is given by $\Z/p^{n-1}$. By \cite[Corollary 7.3]{Mat16}  the weak Galois group of $\mathbf{StMod}(k\Z/p^\infty )$ is the colimit in profinite groups of the sequence:
    \[
    \Z/p \to \Z/p^2 \to \Z/p^3 \to \ldots
    \]
    that agrees with the profinite completion of the Pr\"ufer group $\widehat{\Z/p^\infty}$, that is zero! It then follows that the Galois group is also trivial. Finally, we conclude that the only separable dualizable algebras of finite degree are the finite products of the unit.  
\end{Exa}

\bibliographystyle{alpha}
\bibliography{reference}

\end{document}